\documentclass[12pt]{amsart}
\usepackage{amsmath}
\usepackage{amsthm}
\usepackage{amssymb}
\usepackage{enumerate}
\usepackage{xcolor}
\usepackage{comment}
\newtheorem{theorem}{Theorem}[section]
\newtheorem{lemma}[theorem]{Lemma}
\newtheorem{proposition}[theorem]{Proposition}
\newtheorem{corollary}[theorem]{Corollary}

\newtheorem{remar}[theorem]{Remark}
\theoremstyle{definition}
\newtheorem{example}[theorem]{Example}
\newtheorem{prob}[theorem]{Open Problem}

\newenvironment{remark}{\begin{remar}\rm}{\end{remar}}

\newcommand{\bfind}[1]{\index{#1}{\bf #1}}
\newcommand{\n}{\par\noindent}
\newcommand{\sn}{\par\smallskip\noindent}
\newcommand{\mn}{\par\medskip\noindent}
\newcommand{\bn}{\par\bigskip\noindent}
\newcommand{\pars}{\par\smallskip}
\newcommand{\parm}{\par\medskip}
\newcommand{\parb}{\par\bigskip}
\newcommand{\cal}{\mathcal}

\newcommand{\sep}{^{\rm sep}}
\newcommand{\chara}{\mbox{\rm char}\,}

\newcommand{\Gal}{\mbox{\rm Gal}\,}

\newcommand{\tr}{\mbox{\rm Tr}\,}
\newcommand{\cO}{\mathcal{O}}
\newcommand{\cM}{\mathcal{M}}
\newcommand{\cG}{\mathcal{G}}
\newcommand{\cE}{\mathcal{E}}

\newcommand{\cC}{\mathcal{C}}
\newcommand{\cD}{\mathcal{D}}
\newcommand{\cT}{\mathcal{T}}
\newcommand{\ann}{\mbox{\rm ann}\,}

\newcommand{\fs}{^-}    %{\!\!\uparrow\,}

\newcommand{\R}{\mathbb R}

\newcommand{\N}{\mathbb N}

\newcommand{\F}{\mathbb F}

\begin{document}
\title[Galois extensions with independent defect]{On the computation of K\"ahler
differentials and characterizations of Galois extensions with independent defect}
\author{Steven Dale Cutkosky, Franz-Viktor Kuhlmann and Anna Rzepka}
\date{6.\ 2.\ 2025}

\thanks{The first author was partially supported by grant DMS 2054394
from NSF of the United States.}

\thanks{The second author was partially supported by Opus grant 2017/25/B/ST1/01815 
from the National Science Centre of Poland. He was also supported by a Miller Fellowship
during a four weeks visit to the University of Missouri Mathematics Department in 2022; 
he would like to thank the people at that department for their hospitality.}

\thanks{The authors thank the referee for many helpful corrections and comments.
We also thank Sylvy Anscombe, Arno Fehm, Hagen Knaf and 
Josnei Novacoski for inspiring discussions and helpful feedback.}

\address{Department of Mathematics, University of Missouri, Columbia,
MO 65211, USA}
\email{cutkoskys@missouri.edu}

\address{Institute of Mathematics, University of Szczecin,	
ul. Wielkopolska 15, 	  	  	
70-451 Szczecin, Poland}
\email{fvk@usz.edu.pl}

\address{Institute of Mathematics, University of Silesia in Katowice, Bankowa 14,
40-007 Katowice, Poland}
\email{anna.rzepka@us.edu.pl}

\begin{abstract}\noindent
For important cases of algebraic extensions of valued fields, we develop presentations 
of the associated K\"ahler differentials of the extensions of their valuation rings. 
We compute their annihilators as well as the associated differents. We 
then apply the results to Galois defect extensions of prime degree.
Defects can appear in finite extensions of valued fields of positive residue 
characteristic and are serious obstructions to several problems in positive
characteristic. 
%The fact that perfectoid fields can have defect extensions is a 
%compelling reason to study the defect in detail. In order to do so, a 
A classification of defects (dependent vs.\ independent) has been introduced by the 
second and the third author. It has been shown that perfectoid fields and 
deeply ramified fields only admit extensions with independent defect. We give 
several characterizations of independent defect, using ramification 
ideals, K\"ahler differentials and traces of the maximal ideals of valuation rings.
All of our results are for arbitrary valuations; in particular, we have no restrictions on their rank or value groups. 
\end{abstract}
%\thanks{}


\maketitle
%
%------------------------------------------------------------------------------
%
\section{Introduction}
A central goal of this paper is to give for important cases of algebraic extensions 
of valued fields a presentation of the associated K\"ahler differentials of the 
extensions of their valuation rings. This is crucial not only for applications in the
present paper, but also in the subsequent paper \cite{CuKu}.

By $(L|K,v)$ we denote a field extension $L|K$ where $v$ is a
valuation on $L$ and $K$ is endowed with the restriction of $v$. The valuation
ring of $v$ on $L$ will be denoted by $\cO_L\,$, and that on $K$ by $\cO_K\,$. 
Similarly, $\cM_L$ and $\cM_K$ denote the unique maximal ideals of $\cO_L$ and 
$\cO_K$. The value group of the valued field $(L,v)$ will be denoted by $vL$, and
its residue field by $Lv$. The value of an element $a$ will be denoted by $va$, and 
its residue by $av$. The \bfind{rank} of a valued field $(K,v)$ is the order type of
the chain of proper convex subgroups of its value group $vK$.  All of our results are for arbitrary valuations; in particular, we have no restrictions on their rank or value groups. 
%While all perfectoid
%fields have rank $1$ by definition, deeply ramified fields can have arbitrary ranks. 
Ranks higher than $1$ appear in a natural way when local uniformization, the local form
of resolution of singularities, is studied. Deeply ramified fields of infinite rank 
appear in model theoretic investigations of the tilting construction, as 
presented by Jahnke and Kartas in \cite{JK}.
Therefore, we do not restrict our computations to rank $1$, thereby indicating how 
Kähler differentials and their annihilators, as well as differents, can be 
computed in higher rank.

By $\Omega_{B|A}$  we denote the Kähler differentials, i.e., the module of relative
differentials, when $A$ is a ring and $B$ is an $A$-algebra. In Section~\ref{sect4.1},
we prove:
\begin{theorem}                                    \label{LimProp} 
Let $L|K$ be an algebraic field extension of degree $n$
and suppose that $A$ is a normal domain with quotient field $K$ and $B$ is a domain with
quotient field $L$ such that $A\subset B$ is an integral extension. Suppose that there 
exist generators $b_{\alpha}\in B$ of $L|K$, which are indexed by a totally ordered set 
$S$, such that $A[b_{\alpha}]\subset A[b_{\beta}]$ if $\alpha\le\beta$ and
\[
\bigcup_{\alpha\in S}A[b_{\alpha}]\>=\>B.
\]
Further suppose that there exist $a_{\alpha},a_\beta\in A$ such that $a_{\beta}\mid
a_{\alpha}$ if $\alpha\le \beta$ and for $\alpha\le\beta$, there exist $c_{\alpha,\beta}
\in A$ and expressions 
\begin{equation}                              \label{b_ab_b}
b_{\alpha}\>=\>\frac{a_{\alpha}}{a_{\beta}}b_{\beta}+c_{\alpha,\beta}\>.
\end{equation}
Let $h_{\alpha}$ be the minimal polynomial of $b_{\alpha}$ over $K$. Let $U$ and  $V$ 
be the $B$-ideals 
\[
U\>=\>(a_{\alpha}\mid \alpha\in S)\mbox{ and }V=(h_{\alpha}'(b_{\alpha})\mid \alpha\in S).
\]

Then we have a $B$-module isomorphism
\begin{equation}                      \label{U/UV}
\Omega_{B|A}\>\cong\> U/UV\>.
\end{equation}
\end{theorem}

For the case where $(L|K, v)$ is a valued field extension satisfying condition~(\ref{as}) 
below and $A=\cO_K$ and $B=\cO_L\,$, we compute $V$ in (\ref{altpresV}) and 
for arbitrary $\gamma\in S$, we obtain a $B$-module isomorphism
\begin{equation}                 \label{altpres}
\Omega_{\cO_L|\cO_K}\>\cong\> U/b_\gamma^\dagger U^n \quad \mbox{ with }\quad
b_\gamma^\dagger\,:=\, \frac{h'_\gamma(b_{\gamma})}{a_\gamma^{n-1}}\>.
\end{equation}

We determine the annihilator of $U/UV$ and thus of $\Omega_{\cO_L|\cO_K}$ in 
Proposition~\ref{propann}. For information on the different $\cD(\cO_L|\cO_K)$, 
see Theorems~\ref{Dd} and~\ref{prop}.

\pars
We define the defect and defect extensions 
in Section~\ref{sectdef}. For the construction of a large number of different defect
extensions and the role the defect plays in deep open problems in positive 
characteristic, see \cite{Ku31}. 
The defect can be understood through jumps in the construction of limit key 
polynomials and in the construction of pseudo-convergent sequences. A couple of recent
references explaining this phenomenon are \cite{NS} and \cite{dNS}.

We are generally interested in the study of defect
extensions of arbitrary finite degree. As explained in Section~\ref{sectdef},
via ramification theory this can be reduced to the investigation 
of purely inseparable extensions and of Galois extensions of degree $p=\chara Kv>0$. 

We find explicit realizations for the assumptions of Theorem~\ref{LimProp} for Galois 
defect extensions of prime degree and deduce the next theorem by combining 
Theorems~\ref{OASdef} and ~\ref{OKumdef}. For the definition of higher ramification 
groups and ramification ideals, see Section~\ref{secthrg}.
\begin{theorem}                                    \label{MT2} 
Let $\cE=(L|K,v)$ be a Galois defect extension of prime degree $p$. If $\chara K=0$,
then assume that $K$ contains all $p$-th roots of unity. Then $\cE$ has a 
unique ramification ideal $I_\cE\,$, and there is an $\cO_L$-module isomorphism 
\begin{equation}
\Omega_{\cO_L|\cO_K}\>\cong \> I_\cE / I_\cE^p\>. 
\end{equation}
\end{theorem}
In \cite{NS2}, Novacoski and Spivakovsky use the theory of key polynomials to derive a
presentation of $\Omega_{\cO_L|\cO_K}$ for finite pure extensions $(L|K,v)$ under the 
condition $vL=vK$ (which holds for all extensions covered by the above theorem).
The presentation of $\Omega_{\cO_L|\cO_K}$ for Galois defect extensions $(L|K,v)$ is 
also studied by Thatte in \cite{Th1,Th2}.

\pars
Galois defect extensions of degree $p$ of valued fields of characteristic $p>0$ 
(valued fields of \bfind{equal positive characteristic}) have first been classified 
by the second 
author in \cite{Ku30}. In \cite{KuRz} the classification was extended to the case of
Galois defect extensions of degree $p$ of valued fields of characteristic $0$ with 
residue fields of characteristic $p>0$ (valued fields of \bfind{mixed characteristic}), 
as follows. Take a Galois defect extension $\cE=(L|K,v)$ of prime degree $p$. For every 
$\sigma$ in its Galois group $\Gal (L|K)$, with $\sigma\ne\,$id, we set
\begin{equation}                        \label{Sigsig}
\Sigma_\sigma\>:=\> \left\{ v\left( \left.\frac{\sigma b-b}{b}\right) \right| \, 
b\in L^{\times},\,\sigma b\ne b \right\}\>.
\end{equation}
This set is a final segment of $vL=vK$ and independent of the choice of $\sigma$; 
%(see \cite[Theorems~3.4 and~3.5]{KuRz}); 
we denote it by $\Sigma_\cE\,$. 
%According to Theorem~\ref{dist_galois_p} below, $\Sigma_\cE$ is the unique
%\bfind{ramification jump} of $\cE$ and that 
The $\cO_L$-ideal $I_{\cE}:= (a\in L\mid va\in\Sigma_\cE)$ is the unique 
\bfind{ramification ideal} of $\cE$. 
%(for definitions, see Section~\ref{secthrg}). 
Detailed information on $\Sigma_\cE$ and $I_{\cE}$ is given in
Sections~\ref{secthrg} and~\ref{sectGaldefdegp}.

We say that $\cE$ has \bfind{independent defect} if
\begin{equation}                                     \label{indepdef}
\left\{\begin{array}{lcr}
\Sigma_{\cE}\!\!&=&\!\!\! \{\alpha\in vL\mid \alpha >H\}\>\mbox{ for some proper 
convex subgroup $H$}\\
&&\!\!\! \mbox{ of $vL$ such that $vL/H$ has no smallest positive element.}
\end{array}\right.
\end{equation} 
If (\ref{indepdef}) holds, then we will write 
$H_\cE$ for $H$. If there is no such subgroup $H$, we will say that $\cE$ has
\bfind{dependent defect}. If $vL=vK$ is 
archimedean (i.e., order isomorphic to a subgroup of $\R$), then condition
(\ref{indepdef}) just means that $\Sigma_{\cE}$ consists of all positive elements 
in $vL$ and $vL$ has no smallest positive element. 

\pars
In \cite{GR}, Gabber and Ramero define deeply ramified fields $(K,v)$ by the property
\begin{equation}                         \label{GRdefdr}
\Omega_{\cO_{K\sep}|\cO_K} \>=\> 0\>,
\end{equation}
where $K\sep$ denotes the separable-algbraic closure of $K$. For a valuative definition,
which they prove to be equivalent to theirs, and a generalization thereof, see e.g.\
\cite{KuRz}. In \cite{CuKu} we 
present a simplified version of their proof. The following is a consequence of 
the more general part 1) of Theorem~1.10 in \cite{KuRz}:
\begin{theorem}                                 \label{drinddef}
Every Galois defect extension of prime degree of a deeply ramified field has 
independent defect. 
\end{theorem}

Property (\ref{GRdefdr}) together with Theorem~\ref{drinddef} led us to investigate 
the connection between Kähler differentials and independent defect. In the present 
paper we apply the results stated at the beginning of the introduction to the 
following situation:
\begin{equation}                   \label{as}
\left\{
\begin{array}{l}
(L|K,v) \mbox{ is a Galois defect extension of prime degree } p=\chara Kv\\
\mbox{ and if } \chara K=0\,, \mbox{ then $K$ contains all $p$-th roots of unity.}
\end{array}
\right.
\end{equation}
In this situation we have $vL=vK$, and $L|K$ is an Artin-Schreier extension if 
$\chara K=p$, and a Kummer extension if $\chara K=0$. We obtain that their defect is 
independent if and only if $\Omega_{\cO_L|\cO_K}=0$; this is the 
equivalence of assertions a) and d) in Theorem~\ref{eqindep} below. 

Our results for extensions of the form (\ref{as}) will be complemented in 
%the subsequent paper 
\cite{CuKu} by applying Theorem~\ref{LimProp} to Galois extensions of
prime degree without defect. Using this together with the results of the present paper,
it is shown that a valued field is deeply ramified if and only if the Kähler 
differentials of all Galois extensions of prime degree are zero; moreover, we compute 
Kähler differentials for all finite Galois extensions satisfying (\ref{as}).

\pars
Another way to characterize independent defect of Galois extensions $(L|K,v)$ 
of prime degree is through the trace $\tr_{L|K} \left(\cM_L\right)$ where  
$\tr_{L|K}$ denotes the trace of $L|K$. Note that if $(L|K,v)$ is unibranched, 
then $\tr(\cO_L)\subseteq\cO_K\,$. 

\pars
The following theorem summarizes various characterizations of independent defect. 
The convex subgroups $H$ of $vL$ are in one-to-one correspondence with the
coarsenings $v_H$ of $v$ on $L$ in such a way that $v_H L=vL/H$. The maximal ideal of 
the valuation ring $\cO_{v_H}$ of $v_H$ on $L$ is $\cM_{v_H}=(a\in L \mid va>H)$. If 
$H$ is a subgroup of some ordered abelian group $\Gamma$, then we call it a 
\bfind{strongly convex subgroup} if it is a proper convex subgroup such that $\Gamma/H$
has no smallest positive element. 

If $\chara K=0$ and
$\chara Kv=p>0$, then we denote by $(vK)_{vp}$ the smallest convex subgroup of $vK$ 
that contains $vp$ (cf.\ \cite{KuRz}). If $\chara K>0$, then we set $(vK)_{vp}=vK$. 

\begin{theorem}                              \label{eqindep}
Assume that the extension $\cE=(L|K,v)$ satisfies (\ref{as}). 
Then the following assertions are equivalent:
\sn
a) $\cE$ has independent defect,
\sn
b) the ramification ideal $I_\cE$ of $\cE$ is equal to $\cM_{v_H}$
for some strongly convex subgroup $H$ of $vL$,
\sn
c) $I_\cE^p=I_\cE\,$,
\sn
d) $\Omega_{\cO_L|\cO_K}=0$,
\sn
e) $\tr_{L|K} \left(\cM_L\right)=\cM_{v_H}\cap K$ for some strongly convex subgroup 
$H$ of $vK$.
\parm
If assertions b) or e) hold, then $H=H_\cE$, and the valuation ring of $v_H$ is
the localization of $\cO_L$ with respect to the ramification
ideal $I_\cE$. The value group of the corresponding coarsening of $v$ does not have 
a smallest positive element.

\pars
If $vK$ is archimedean, or more generally, if $(vK)_{vp}$ is archimedean, 
then $H_\cE$ can only be equal to $\{0\}$ and if the
ramification ideal is a prime ideal, then it can only be equal to $\cM_L\,$.

\pars
If $\Omega_{\cO_L|\cO_K}$ is annihilated by $\cM_L$, then $\Omega_{\cO_L|\cO_K}=0$.
\end{theorem}
For more equivalent conditions, see part 1) of Proposition~\ref{ind=I^p=I}.

The equivalence of assertions a) and b) follows from the definition of independent
defect and the fact that $I_{\cE}= (a\in L\mid va\in\Sigma_\cE)$.
The equivalence of assertions a) and c) will be proved in part 1) of 
Proposition~\ref{ind=I^p=I}.
The equivalence of assertions c) and d) follows from Theorem~\ref{MT2}.
The equivalence of assertions a) and e) follows from the next theorem. 
The last assertion follows from part 3) of Proposition~\ref{propann}.
The remaining assertions, in particular the equality $H=H_\cE$, will be proven 
together with the next theorem in Section~\ref{secttracepf}.
\begin{theorem}                                          \label{defetrace}
Assume that the extension $\cE=(L|K,v)$ satisfies (\ref{as}). Then
\begin{equation}                                          \label{defetraceeq}
\tr_{L|K} \left(\cO_L\right) \>=\> \tr_{L|K} \left(\cM_L\right) \>=\> 
(b\in K \mid vb\in (p-1)\Sigma_\cE) \>=\> (I_\cE\cap K)^{p-1}\>.
\end{equation}
The extension $\cE$ has independent defect if and only if for some 
strongly convex subgroup $H$ of $vL$,
\begin{equation}                                            \label{trMM}
\tr_{L|K} \left(\cO_L\right) \>=\> \tr_{L|K} \left(\cM_L\right) \>=\> \cM_{v_H}\cap K 
\>=\> (b\in K\mid vb> H)\>.
\end{equation}
In particular, if $H=\{0\}$ (which is always the case 
if $vK$ is archimedean, or more generally, if $(vK)_{vp}$ is archimedean), 
then this means that
\[
\tr_{L|K} \left(\cM_L\right) \>=\> \cM_K\>.
\]
In the mixed characteristic case, if (\ref{trMM}) holds, then $\cM_{v_H}$ will contain
$p$, so that $\chara Kv_H =p$.
\end{theorem}

\parm
In \cite{CG} the authors introduce the notion of {\it deeply ramified extensions}
$(L|K,v)$ where $(K,v)$ is a local field and $L|K$ is algebraic. 
%The authors of \cite{GR} mention that their definition of the notion {\it deeply 
%ramified field} is based on this notion. 
It follows from
\cite[Proposition 2.9]{CG} that under the conditions above, $(L|K,v)$ is a
deeply ramified extension if and only if $\tr_{F|L} \left(\cM_F\right)=\cM_L$ for
every finite extension $F|L$ (note that as an algebraic extension of a local field, 
$(L,v)$ is henselian, and so the extension of $v$ to $F$ is uniquely determined); 
see also \cite[Theorem~1.1]{F}. Theorem~\ref{defetrace} shows that the equality  
$\tr_{F|L} \left(\cM_F\right)=\cM_L$
will in general not characterize deeply ramified fields, unless they have rank 1
(as is the case for algebraic extensions of local fields). 

\pars
The \bfind{different} of $(L|K,v)$ is $\cD(\cO_L|\cO_K):=\cO_L:_L \cC(\cO_L|\cO_K)$,
where $\cC(\cO_L|\cO_K):=(z\in L\mid \tr(z\cO_L)\subseteq\cO_K)$ is the fractional
$\cO_L$-ideal called the \bfind{complementary ideal}; see \cite[Ch.~V, \S11]{ZSI}. For
$b\in \cO_L$ and $h_b$ its minimal polynomial over $K$, the element $h'_b(b)\in \cO_L$ is
called the \bfind{different} of $b\,$.

We denote the annihilator of an 
%nonzero 
$\cO_L$-ideal $I$ by $\ann I$. 
In Section~\ref{sectDd}, we will prove:
\begin{theorem}                             \label{Dd}
Assume that the extension $\cE=(L|K,v)$ satisfies (\ref{as}). 
\sn
1) The following $\cO_L$-ideals are equal:
\n
i) \ $I_\cE^{p-1}$,
\n
ii) \ the ideal $V$ from Theorem~\ref{LimProp}  when applied to $L|K$ with $A=\cO_K$ and
$B=\cO_L\,$,
\n
iii) \ the ideal generated by the differents of all elements of $\cO_L
\setminus\cO_K\,$,
\n
iv) \ the ideal generated by the elements of $\tr_{L|K} \left(\cO_L\right)$.
\sn
2) We have that $\cD(\cO_L|\cO_K)=I_\cE^{p-1}$ if and only if $vI_\cE^{p-1}$ has no 
infimum in $vL$. If $vI_\cE^{p-1}$ has infimum $va$ in $vL$ for some $a\in L$, then 
$\cD(\cO_L|\cO_K)=a\cO_L\ne I_\cE^{p-1}$ and $I_\cE^{p-1}=\cM_L\,\cD(\cO_L|\cO_K)$. 
\sn
3) \ If $(K,v)$ has rank $1$, then $\cD(\cO_L|\cO_K)=\ann\Omega_{\cO_L|\cO_K}\,$.
\sn
4) \ If the extension $\cE$ has independent defect, then $\cD(\cO_L|\cO_K)=
\cO_L=\ann\Omega_{\cO_L|\cO_K}$ if $H_{\cE}=\{0\}$, and $\cD(\cO_L|\cO_K)=
\cM_{v_{H_\cE}}\subsetneq\cO_L=\ann\Omega_{\cO_L|\cO_K}$
otherwise. Conversely, if $\cD(\cO_L|\cO_K)$ is equal to $\cO_L$ or to $\cM_{v_H}$ for a 
nontrivial strongly convex subgroup $H$ of $vL$, then $\cE$ has independent defect.
\end{theorem}

The reader may note that in rank higher than $1$, $\cD(\cO_L|\cO_K)$ is not necessarily
equal to the annihilator of $\Omega_{\cO_L|\cO_K}\,$; this fact has already been observed 
in \cite{Th1}. For the computation of annihilators 
in arbitrary rank, see Section~\ref{sectann}. For the 
different of a finite unibranched extension $(L|K,v)$ with $vL=vK$, see 
Theorem~\ref{prop}. 

\parm
Finally, let us give further characterizations of independent defect, and in particular
show that it can be characterized by a simple condition in $(K,v)$ (which is an 
important point in \cite{JK}).
If $z$ is an element and $S$ a subset of some valued field $(L,v)$, then we set
\[
v(z-S)\>:=\>\{v(z-c)\mid c\in S\}.
\]
We will write $\wp(X)$ for the Artin-Schreier polynomial $X^p-X$.
\begin{theorem}                              \label{eqindep2AS}
Assume that $\cE=(L|K,v)$ is an Artin-Schreier defect extension generated by $\vartheta
\in L$ with $\vartheta^p-\vartheta=a\in K$. Then $v(\vartheta-c)<0$ and $v(a-\wp(c))=
pv(\vartheta-c)$ for all $c\in K$, and therefore,
\begin{equation}                                 \label{1.4eq1}
v(a-\wp(K))\>=\>pv(\vartheta-K)\>. 
\end{equation}
Further, the following are equivalent:
\sn
a) $\cE$ has independent defect,
\sn
b) $v(\vartheta-K)=-\{\alpha\in vK\mid \alpha >H\}$ for some strongly 
convex subgroup $H$ of $vK$,
\sn
c) $v(a-\wp(K))=-\{\alpha\in pvK\mid \alpha >H\}$ for some strongly 
convex subgroup $H$ of $vK$,
\sn
d) there is some strongly convex subgroup $H$ of $vK$ such that for all $b\in K$ such
that $vb>H$ there is $c\in K$ satisfying $v(a-\wp(c))\geq -vb$.
\sn
If $vK$ is $p$-divisible, then these properties are also equivalent to 
\sn
e) $pv(\vartheta-K)=v(\vartheta-K)$,
\sn
f) $v(a-\wp(K))=v(\vartheta-K)$.
\sn
If $(K,v)$ has rank $1$, then condition d) is equivalent to:
\sn
g) for all $b\in \cO_K$ there is $c\in K$ satisfying $v(a-\wp(c))\geq -vb$. 
\end{theorem}

\pars
For the notion of ``$1$-unit'', see Section~\ref{sectASKext}.
\begin{theorem}                              \label{eqindep2K}
Assume that $\cE$ is a Kummer defect extension of prime degree $p$. Then it is 
generated by a $1$-unit $\eta
\in L$ with $\eta^p=a\in K$. We have $v(\eta-c)<\frac{1}{p-1}vp$ and $v(a-c^p)=
pv(\eta-c)$ for all $c\in K$, and therefore,
\begin{equation}                                 \label{1.4eq2}
v(a-K^p)\>=\>pv(\eta-K)\>. 
\end{equation}
Further, the following are equivalent:
\sn
a) $\cE$ has independent defect,
\sn
b) $v(\eta-K)=\frac{1}{p-1}vp-\{\alpha\in vK\mid \alpha >H\}$ for some strongly 
convex subgroup $H$ of $vK$,
\sn
c) $v(a-K^p)=\frac{p}{p-1}vp-\{\alpha\in pvK\mid \alpha >H\}$ for some strongly 
convex subgroup $H$ of $vK$,
\sn
d) there is some strongly convex subgroup $H$ of $vK$ such that for all $b\in K$ such
that $vb>H$ there is $c\in K$ satisfying $v(a-c^p)\geq \frac{p}{p-1}vp-vb$.
\sn
If $vK$ is $p$-divisible, then these properties are also equivalent to 
\sn
e) $pv(\eta-K)=vp+v(\eta-K)$,
\sn
f) $v(a-K^p)=vp+v(\eta-K)$.
\sn
If $(K,v)$ has rank $1$, then condition d) is equivalent to:
\sn
g) for all $b\in \cO_K$ there is $c\in K$ satisfying $v(a-c^p)\geq \frac{p}{p-1}vp-vb$. 
\end{theorem}

\begin{remark}
Assume the situation of this theorem. Since $K$ contains a primitive root 
$\zeta_p$ of unity and $v(1-\zeta_p)=\frac{1}{p-1}vp$, assertion b) is equivalent to  
\sn
b') $v(\frac{\eta}{1-\zeta_p}-K)=-\{\alpha\in vK\mid \alpha >H\}$ for some strongly 
convex subgroup $H$ of $vK$.
\end{remark}

We note that conditions g) in the two theorems are elementary in the language of 
valued rings. It has been pointed out to us by Arno Fehm that in fact also conditions 
d) are elementary, as it can be stated in the language of valued rings that in the
setting of Theorem~\ref{eqindep2AS}, the set 
\[
\{\alpha\in vK\mid v(a-\wp(K))\,<\,\alpha\,<\,-v(a-\wp(K))\} 
\]
and in the setting of Theorem~\ref{eqindep2K}, the set
\[
\left\{\alpha\in vK \>\left|\> v(a-K^p)-\frac{p}{p-1}vp\,<\,\alpha\,<\,-v(a-K^p)
+\frac{p}{p-1}vp\right.\right\} 
\]
is a strongly convex subgroup $H$ of $vK$.

\pars
The following can be deduced from Theorems~\ref{eqindep2AS} and~\ref{eqindep2K}; 
for details, see \cite{KuRz2}. 
Take a deeply ramified field $(K,v)$ of rank $1$ (e.g.\ a perfectoid field) with
$\chara Kv=p>0$. If it is of equal characteristic, then it satisfies the sentence
\[
\forall a,b\in \cO_K \,\exists c\in K:\>v(a-\wp(c))\>\geq\> -vb\>;
\]
likewise, if it is of mixed characteristic, then it satisfies the sentence
\[
\forall a,b\in \cO_K \,\exists c\in K:\>v(a-c^p)\>\geq\> \frac{p}{p-1}vp-vb\>.
\]
Hence the same holds for every elementary extension $(K^*,v^*)$
of $(K,v)$. Assume that $v^*$ admits a proper nontrivial coarsening $w$. Then 
$(K^*,w)$ satisfies the sentence
\[
\forall a\in \cO_{(K^*,w)} \,\exists c\in K^*:\>w(a-\wp(c))\>\geq\> 0
\]
or
\[
\forall a\in \cO_{(K^*,w)} \,\exists c\in K^*:\>w(a-c^p)\>\geq\> \frac{p}{p-1}wp\>,
\]
respectively, which shows that every Artin-Schreier extension or Kummer extension  
of degree $p$, respectively, either lies in the henselization of $(K^*,w)$ or is tame
(see the definition given in \cite{KuRz}); in particular, they are defectless. This 
is also a consequence of results obtained in \cite{JK}, by methods different from 
those used in \cite{KuRz2}.

\parb
In \cite{KuRz2} we will give constructions that will show that all the situations
mentioned in the above theorems can appear. To conclude this introduction,
let us give an example of an interesting extension with independent defect.
\begin{example}
Choose a prime $p$ and let $\F_p$ denote the field with $p$ elements.
Consider the rational function field $\F_p(t)$, equipped with the $t$-adic
valuation $v_t$, and its perfect hull $K=\F_p(t)(t^{1/p^n}\mid n\in \N)$, equipped
with the unique extension of $v_t\,$. Take a root $\vartheta$ of the Artin-Schreier
polynomial
\[
X^p-X-\frac{1}{t}\>.
\] 
The extension $(K(\vartheta)|K,v_t)$ was presented by Shreeram Abhyankar in 
\cite{[Ab1]}. It became famous since it shows that there are elements algebraic over
$\F_p(t)$ with a power series expansion in which the exponents do not
have a common denominator. However, this extension is also an important example for
an extension with nontrivial defect. The situation remains the same if we replace 
$\F_p(t)$ by the field $\F_p((t))$ of formal Laurent series (see 
\cite[Example 3.12]{Ku31}). 

In both cases, $(K,v_t)$ is a deeply ramified field, as it is perfect of positive
characteristic. Hence by Theorem~\ref{drinddef}, the extension $(K(\vartheta)|K,v_t)$
has independent defect. Now Theorem~\ref{eqindep} shows that 
$\Omega_{\cO_{K(\vartheta)}|\cO_K}=0$.
\end{example}

\bn
%
%------------------------------------------------------------------------------
%
\section{Preliminaries}
%
%------------------------------------------------
%
\subsection{Defect}                      \label{sectdef}
\mbox{ }\sn
A valued field extension $(L|K,v)$ is \bfind{unibranched} if the
extension of $v$ from $K$ to $L$ is unique. Note that a unibranched extension is
automatically algebraic, since every transcendental extension always admits several
extensions of the valuation.

If $(L|K,v)$ is a finite unibranched extension, then by the Lemma of Ostrowski
\cite[Corollary to Theorem~25, Section G, p.~78]{ZS2}),
\begin{equation}                    \label{feuniq}
[L:K]\>=\> \tilde{p}^{\nu }\cdot(vL:vK)[Lv:Kv]\>,
\end{equation}
where $\nu$ is a non-negative integer and $\tilde{p}$ the
\bfind{characteristic exponent} of $Kv$, that is, $\tilde{p}=\chara Kv$ if it is 
positive and $\tilde{p}=1$ otherwise. The factor $d(L|K,v):=\tilde{p}^{\nu }$ is 
the \bfind{defect} of the extension $(L|K,v)$. We call $(L|K,v)$ a \bfind{defect
extension} if $d(L|K,v) >1$, and a \bfind{defectless extension} if $d(L|K,v)=1$.
Nontrivial defect only appears when $\chara Kv=p>0$, in which case $\tilde{p}=p$.
A valued field $(K,v)$ is \bfind{henselian} if it satisfies Hensel's Lemma, or 
equivalently, if all of its algebraic extensions are unibranched.
A henselian field $(K,v)$ is called a \bfind{defectless field} if all of its finite
extensions are defectless.

Throughout this paper, when we talk of a \bfind{defect extension $(L|K,v)$ of prime
degree}, we will always tacitly assume that it is a unibranched extension. Then it
follows from (\ref{feuniq}) that $[L:K]=p=\chara Kv$ and that $(vL:vK)=1=[Lv:Kv]$; 
the latter means that $(L|K,v)$ is an \bfind{immediate extension}, i.e.,
the canonical embeddings $vK\hookrightarrow vL$ and $Kv\hookrightarrow Lv$ are onto.

In order to reduce arbitrary finite defect extensions to purely inseparable extensions 
and Galois extensions of degree $p=\chara Kv>0$, we fix an extension of $v$ from $K$ 
to its algebraic closure $\tilde{K}$. 
%We denote the separable-algbraic closure of $K$ by $K\sep$. 
The \bfind{absolute ramification field of $(K,v)$} (with respect to the chosen extension 
of $v$), denoted by $(K^r,v)$, is the ramification field of the normal extension 
$(K\sep|K,v)$. If $a\in\tilde{K}$ is such that $(K(a)|K,v)$ is a defect extension, then 
$(K^r(a)|K^r,v)$ is a defect extension with the same defect (see 
\cite[Proposition~2.12]{KuRz}). On the other hand, $K\sep|K^r$ is a $p$-extension (see 
\cite[Lemma 2.7]{Ku30}), so $K^r(a)|K^r$ is a tower of
purely inseparable extensions and Galois extensions of degree $p$.

\mn
%
%------------------------------------------------
%
\subsection{Final segments}                      \label{sectcdist}
\mbox{ }\sn
We recall basic notions and facts connected with final segments in ordered abelian 
groups, and how they relate to elements in valued field extensions. For the details and
proofs see Section 2.3 of \cite{Ku30} and Section 3 of \cite{KuVl}.

Take a totally ordered set $(T,<)$. For a nonempty subset $S$ of $T$ and an element 
$t\in T$ we will write $t<S$ if $t<s$ for every $s\in S$, and similarly for $>$ in
place of $<$. A set $S\subseteq T$ is called 
an \bfind{initial segment} of $T$ if for each $s\in S$ every $t <s$ also lies in $S$.
Similarly, $S\subseteq T$ is called a \bfind{final segment} of $T$ if for each $s\in S$ 
every $t >s$ also lies in $S$. By definition, final and initial segments are convex. If 
$S$ is a final segment of $T$, then $T\setminus S$ is an initial segment, and vice versa. 

For a subset $M$ of $T$ we define 
\[
M^+\>:=\> \{t\in T\mid t>M\} \quad\mbox{ and }\quad M^-\>:=\> \{t\in T\mid \exists\, 
m\in M \,\,t\geq m\}\>.
\]

Take a subset $S$ in an ordered abelian group $\Gamma$, an element $\gamma\in\Gamma$, and 
$n\in\N_{>0}\,$. We set $-S:=\{-s\mid s\in S\}$, $\gamma+S:=\{\gamma+s\mid s\in S\}$, and 
$nS:=\{ns\mid s\in S\}$. We note that if $S$ is a final segment of $\Gamma$, then $-S$ is 
an initial segment of $\Gamma$, and $\gamma+S$ is a final segment of $\Gamma$. If
$\Gamma$ is divisible by $n$ and $S$ is a final segment of $\Gamma$, then also 
$nS$ is a final segment of $\Gamma$. Further, we note that
\begin{equation}                    \label{g+cut}
(\gamma+S)^+ \>=\>\gamma+S^+ \quad\mbox{ and }\quad (\gamma+S)^- \>=\>\gamma+S^-\>.
\end{equation}

\mn
%
%------------------------------------------------------------------------------
%
\subsection{Strongly convex subgroups}     \label{sectcsac}
\mbox{ }\sn
%
%If $S$ is a subset of a totally ordered set $T$, then by $S\fs$ we denote the smallest
%final segment of $T$ that contains $S$. Recall that if in addition $T$ is an ordered 
%abelian group, then $nS=\{ns\mid s\in S\}$. 

\begin{lemma}                           \label{SD}
Take an ordered abelian group $(\Gamma,<)$, a nonempty final segment $\Sigma$ of 
$\Gamma^{\geq 0}:=\{\gamma\in \Gamma\mid \gamma\geq 0\}$, and $m\in\N$, $m\geq 2$. 
Assume that $\Sigma\ne\Gamma^{\geq 0}$. Then the following assertions are equivalent:
%\sn
%a) $\Delta:=\{\gamma\in\Gamma\mid \gamma,-\gamma\notin\Sigma\}$ is a proper convex 
%subgroup of $\Gamma$ such that the ordered abelian group $\Gamma/\Delta$ has no 
%smallest positive element,
\sn
a) there is a strongly convex subgroup $\Delta$ of $\Gamma$ such that $\Sigma=
\Gamma^{\geq 0}\setminus\Delta$,
%\sn
%c) for some $n\geq 2$, if $\alpha\in\Sigma$, then there is $\beta\in\Sigma$ such that
%$n\beta\leq\alpha$. 
\sn
b) $\Sigma=(m\Sigma)\fs$.
\end{lemma}
\begin{proof}
a)$\Rightarrow$b): Since $\Sigma$ is a final segment of $\Gamma^{\geq 0}$, we know that 
$(m\Sigma)\fs\subseteq\Sigma$. Hence in order to prove that $\Sigma=(m\Sigma)\fs$, we 
have to show the reverse inclusion. For this, we will show that
for every $\alpha\in\Sigma$ there is $\beta\in\Sigma$ such that $m\beta<\alpha$. 
The equality $\Sigma=\Gamma^{\geq 0}\setminus\Delta$ implies that $\Sigma=\{\gamma\in 
\Gamma\mid \gamma >\Delta\}$, hence we have $\alpha+\Delta>0$ in $\Gamma/\Delta$. 
As $\Gamma/\Delta$ has no smallest positive element, the positive elements cannot be 
bounded away from $0$ by an element in the divisible hull of $\Gamma/\Delta$, and so 
there must be $\beta\in\Gamma$ such that $0<\beta+\Delta<\frac 1 m (\alpha+\Delta)$. 
This implies $m\beta +\Delta<\alpha+\Delta$, whence $m\beta<\alpha$, as desired. 
\sn
b)$\Rightarrow$a): Since $\Sigma\ne\Gamma^{\geq 0}$ is a nonempty final segment of 
$\Gamma^{\geq 0}$, we see that $\Gamma^{\geq 0}\setminus \Sigma$ is nonempty and
convex. We prove that $\Gamma^{\geq 0}\setminus \Sigma$ is additively closed. Take 
$\gamma,\gamma'\in \Gamma^{\geq 0}\setminus\Sigma$; we may assume that $\gamma'\leq
\gamma$. Hence $0\leq\gamma+\gamma'\leq m\gamma$. Suppose that $m\gamma\in\Sigma$. Then 
by the statement of b), there is $\beta\in\Sigma$ such that $m\beta\leq m\gamma$,
whence $\beta\leq\gamma$. From the convexity of $\Sigma$ it follows that $\gamma\in
\Sigma$, contradiction. Thus $m\gamma\in\Gamma^{\geq 0}\setminus \Sigma$, and by the
convexity of $\Gamma^{\geq 0}\setminus \Sigma$, we obtain that
$\gamma+\gamma'\in\Gamma^{\geq 0}\setminus \Sigma$. This
proves our claim.

Since $\Gamma^{\geq 0}\setminus \Sigma$ is convex, also $\Delta:=\{\gamma,-\gamma
\mid \gamma\in\Gamma^{\geq 0}\setminus \Sigma\}$ is convex. Since $\Gamma^{\geq 0}
\setminus \Sigma$ is additively closed, it follows by convexity that $\Delta$ is 
additively closed, hence a convex subgroup of $\Gamma$.

Take any $\alpha\in\Gamma$ such that $0<\alpha+\Delta$. It follows that
$0<\alpha\notin\Delta$, hence $\alpha\in\Sigma$. Take $\beta\in\Sigma$ such that
$m\beta\leq\alpha$. We have that $\beta+\Delta<m\beta+\Delta$ since otherwise, 
$(m-1)\beta\in\Delta$ and by convexity, $\beta\in \Delta$, contradiction. Thus 
$0<\beta+\Delta<m\beta+\Delta\leq \alpha+\Delta$. This shows that $\alpha+\Delta$
is not the smallest positive element of $\Gamma/\Delta$. As $\alpha+\Delta$ was 
arbitrary, we conclude that $\Gamma/\Delta$ has no smallest positive element.
\end{proof}

\mn
%
%
%------------------------------------------------------------
%
\subsection{Immediate and unibranched extensions}        \label{sectimmdeflext}
\mbox{ }\sn
Take a valued field $(K,v)$. A \bfind{henselization} of $(K,v)$ is an algebraic 
extension of $(K,v)$ which admits a valuation preserving embedding in every other 
henselian extension of $(K,v)$. Henselizations always exist and are unique up to 
valuation preserving isomorphism over $K$; therefore we will talk of {\it the} 
henselization of $(K,v)$ and denote it by $(K,v)^h=(K^h,v^h)$. The henselization of 
$(K,v)$ is an immediate separable-algebraic extension.
\pars
In what follows, we will need the following result, which is \cite[Lemma 2.1]{BlKu}:
\begin{lemma}	                 	\label{lindishens}
A finite extension $(L|K,v)$ is unibranched if and only if $L|K$ is linearly disjoint 
from some (equivalently, every) henselization of $(K,v)$.
\end{lemma}

Let us consider an immediate but not necessarily algebraic extension   %\linebreak 
$(K(x)|K,v)$. Then by \cite[Theorem 2.19]{Ku30} the set $v(x-K)\subseteq vK$ is a
final segment of $vK$; in particular, it
has no maximal element. If $g\in K[X]$ and there is $\alpha\in v(x-K)$ such that 
for all $c\in K$ with $v(x-c)\geq \alpha$ the value $vg(c)$ is constant, then we will 
say that \bfind{ the value of $g$ is ultimately fixed over $K$}. We call $(K(x)|K,v)$ 
\bfind{pure in $x$} if the value of every $g(X)\in K[X]$ of degree smaller than 
$[K(x):K]$ is ultimately fixed over $K$. Note that we set
$[K(x):K]=\infty$ if $x$ is transcendental over $K$.
\begin{lemma}                       \label{lp}
Every unibranched immediate extension $(K(x)|K,v)$ of prime degree is pure in $x$. 
\end{lemma}
\begin{proof}
By the Lemma of Ostrowski, every such extension must have a defect equal to its degree, 
which must be equal to its residue characteristic $p>0$. 

In the case of a henselian field $(K,v)$, the assertion follows from \cite[Proposition 
6.5]{KuVl}. We have to consider the general case. Since $(K(x)|K,v)$ is unibranched, 
Lemma~\ref{lindishens} shows that $K(x)|K$ is linearly disjoint from $K^h|K$. Hence 
$(K^h(x)|K^h,v)$ is a unibranched extension of degree $p$, and it is immediate since 
$K^h(x)=K(x)^h$ and henselizations are immediate extensions. By what we have said above,
the extension $(K^h(x)|K^h,v)$ is pure.

Further, it follows from \cite[Theorem 2]{Ku34} that there is no element $b\in K^h$ such 
that $v(x-b)>v(x-K)$. Thus, if there were a polynomial $g\in K[X]$ of degree less than 
$p$ whose value is not ultimately fixed over $K$, then its value 
would also not be ultimately fixed over $K^h$, contradiction.
\end{proof}

The next lemma follows from \cite[Lemma~8]{Ka} and \cite[Lemma~5.2]{KuVl}. Note that if 
$(K(x)|K,v)$ is an extension such that $v(x-K)$ has no maximal element, then by the proof 
of \cite[Theorem 1]{Ka}, $x$ is limit of a pseudo Cauchy sequence in $(K,v)$ without 
limit in $K$, or equivalently, by \cite[part a) of Lemma~4.1]{KuVl} its approximation 
type over $(K,v)$ is immediate. We use the Taylor expansion
\begin{equation}                             \label{Taylorexp}
f(X) = \sum_{i=0}^{n} \partial_i f(c) (X-c)^i
\end{equation}
where $\partial_i f$ denotes the $i$-th \bfind{Hasse-Schmidt derivative} of $f$.
\begin{lemma}                                \label{Kap}
Take an immediate extension $(K(x)|K,v)$ that is pure in $x$. Let $p$ be the 
characteristic exponent of $Kv$. Then for every nonconstant polynomial $f\in K[X]$ of 
degree smaller than $[K(x):K]$ there is some $\gamma\in v(x-K)$ such that for all
$c\in K$ with $v(x-c)\geq \gamma$ and all $i$ with $1\leq i\leq\deg f$, \sn
the values $v\partial_i f(c)$ are fixed, equal to $v\partial_i f(x)$, and\sn
the values $v\partial_i f(c)+i\cdot v(x-c)$ are pairwise distinct.
\end{lemma}

\mn
%
%
%------------------------------------------------------------
%
\subsection{Artin-Schreier and Kummer extensions}        \label{sectASKext}
\mbox{ }\sn
Every Galois extension of degree $p$ of a field $K$ of characteristic $p>0$ is an
\bfind{Artin-Schreier extension}, that is, generated by an \bfind{Artin-Schreier
generator} $\vartheta$ which is the root of an \bfind{Artin-Schreier polynomial} 
$X^p-X-b$ with $b\in K$. For every $c\in K$, also $\vartheta-c$ is an Artin-Schreier
generator as its minimal polynomial is $X^p-X-b+c^p-c$. Every Galois extension 
of prime degree $n$ of a field $K$ of characteristic different from $n$ which contains 
all $n$-th roots of unity is a \bfind{Kummer
extension}, that is, generated by a \bfind{Kummer generator} $\eta$ which satisfies 
$\eta^n\in K$. For these facts, see \cite[Chapter VI, \S6]{[L]}.

%More generally, a \bfind{radical extension} of prime degree $n$ of a field $K$ is an
%extension of $K$ generated by an element $\eta$ such that $\eta^n\in K$; it need not 
%be a Galois extension. In this case we call $\eta$ a \bfind{radical generator} of the
%extension. 

A \bfind{1-unit} in a valued field $(K,v)$ is an element of the form $u=1+b$ with 
$b\in\cM_K\,$; in other words, $u$ is a unit in $\cO_K$ with residue 1. Take a Kummer 
extension $(L|K,v)$ of degree $p$ of fields of characteristic 0 with any Kummer generator 
$\eta$. Assume that $v\eta\in vK$, so that there is $c_1\in K$ such that $vc_1=-v\eta$, 
whence $vc_1\eta =0$. Assume further that $c_1\eta v\in Kv$, so that there is $c_2\in 
K$ such that $c_2v=(c_1\eta v)^{-1}$. Then $vc_2c_1\eta=0$ and $c_2c_1\eta v=1$.
Furthermore, $K(c_2c_1\eta)=K(\eta)$ and $(c_2c_1\eta)^p=
c_2^pc_1^p\eta^p\in K$. Hence $c_2c_1\eta$ is a Kummer generator of $(L|K,v)$
and a $1$-unit. This shows that in particular, if $(L|K,v)$ is an immediate Kummer
extension, it always admits a Kummer generator which is a $1$-unit.

We note that if $\eta$ is a $1$-unit and if $v(\eta-c)>v\eta
=0$, then also $c$ and $c^{-1}$ are 1-units; conversely, if $c$ is a 1-unit, 
then $v(\eta-c)>0$. 

\pars
We will need the following facts.
\begin{lemma}
Take a valued field $(K,v)$, $2\leq n\in\N$, and a primitive 
$n$-th root of unity $\zeta\in K$. Then 
\begin{equation}                       \label{prod1-z}
\prod_{i=1}^{n-1} (1-\zeta^i)\>=\> n \>.
\end{equation}
If in addition $n$ is prime, then
\begin{equation}                       \label{vz-1}
v(1-\zeta)\>=\>\frac{vn}{n-1} \>.  
\end{equation}
\end{lemma}
\begin{proof}
Let $f(X)=X^n-1=\prod_{i=1}^n (X-\zeta^i)$. Then
$f'(X)=nX^{n-1}$ and  $\prod_{i=1}^{n-1} (1-\zeta^i)=f'(1)=n$. This proves
equation~(\ref{prod1-z}). 

If $n$ is prime, then all powers $\zeta^k\ne 1$ are also primitive $n$-th roots of unity. 
We may assume that $\zeta$ is chosen among the primitive $n$-th roots of unity such that
$v(1-\zeta)$ is maximal. Take any $k\in\N$ such that $\zeta^k\ne 1$. Then
\[
1-\zeta^k\>=\> (1-\zeta)(1+\zeta+\ldots+\zeta^{k-1})\>. 
\]
Since $v\zeta=0$, we have $v(1+\zeta+\ldots+\zeta^{k-1})\geq 0$, whence $v(1-\zeta^k)\geq
v(1-\zeta)$. By our choice of $\zeta$, this yields $v(1-\zeta^k)=v(1-\zeta)$. 
Consequently, all factors of the product in (\ref{prod1-z}) have equal value. This proves
equation~(\ref{vz-1}). 
\end{proof}

\mn
%
%------------------------------------------------
%
\subsection{Ideals and final segments}           %\label{}
\mbox{ }\sn
Take a valued field $(L,v)$. The function
\begin{equation}                           \label{vIS}
v:\; I\>\mapsto\> vI\>:=\> \{vb\mid 0\ne b\in I\}
\end{equation}
is an order preserving bijection from the set of all proper ideals of ${\cal O}_L$ onto 
the set of all final segments of $vL^{> 0}$ (including the final segment $\emptyset$). 
The set of these final segments is linearly ordered by
inclusion, and the function (\ref{vIS}) is order preserving: $J\subseteq I$ holds 
if and only if $vJ\subseteq vI$ holds. The inverse of the above function 
is the order preserving function
\begin{equation}                           \label{S->IS}
\Sigma\>\mapsto\> I_\Sigma\>:=\> (a\in L\mid va\in \Sigma)\>=\>
\{a\in L\mid va\in \Sigma\}\cup\{0\}\>.
\end{equation}

The following facts from general valuation theory are well known:
\begin{lemma}
The following statements are equivalent:
\sn
a) $I_\Sigma$ is a prime ideal of $\cO_L\,$,
\sn
b) $\Sigma= vL^{\geq 0}\setminus H$ for some convex subgroup $H$ of $vL$.
\sn
If b) holds, then $I_\Sigma=\cM_{v_H}$ for the 
coarsening $v_H$ of $v$ on $L$ whose value group is $vL/H$.
\end{lemma}

We note that for each $S\subset vL^{\geq 0}$, we have the following equalities of 
$\cO_L$-ideals:
\begin{equation}                      \label{17}
(a\in L\mid va\in S)\>=\> (a\in L\mid va\in S\fs)\>=\>I_{S^-}\>.
\end{equation}
\begin{lemma}                         \label{lemI^n=I}
Take a nonempty final segment $\Sigma$ of $vL^{\geq 0}$ and $m\in\N_{>0}\,$.
\sn
1) We have
\[
I_\Sigma^m\>=\> (b\in L\mid vb\in m\Sigma) \>=\> (b\in L\mid vb\in (m\Sigma)\fs)\>.
\]

\sn
2) We have $I_\Sigma=I_\Sigma^m$ if and only if $\Sigma=(m\Sigma)\fs$.
\end{lemma}
\begin{proof}
1): We have 
\begin{eqnarray*}                
I_\Sigma^m &=& (a^m\mid a\in I_\Sigma)\>=\> (a^m\mid a\in L \mbox{ and }va\in \Sigma)
\>=\> (b\in L\mid vb\in m\Sigma)\\
&=& (b\in L\mid vb\in (m\Sigma)\fs)\>,
\end{eqnarray*}
where the fourth equality follows from (\ref{17}),
and the third equality is seen as follows. The inclusion $\subseteq$ is obvious.
Assume that $b\in L$ with $vb\in m\Sigma$, i.e., there is $\alpha\in\Sigma$ such
that $m\alpha=vb$. Choose $a\in L$ such that $va=\alpha$. Then $va^m=vb$, hence
$b\in (a^m\mid a\in L \mbox{ and }va\in \Sigma)$. This proves the inclusion $\supseteq$.

\sn
2): This follows immediately from part 1).
\end{proof}

\mn
%
%------------------------------------------------
%
\subsection{Ramification jumps and ramification ideals}           \label{secthrg}
\mbox{ }\sn
Take a valued field $(K,v)$. Assume that $L|K$ is a Galois extension, and let
$G=\Gal(L|K)$ denote its Galois group. For proper ideals $I$ of ${\cal O}_L$ we consider 
the (upper series of) \bfind{higher ramification groups}
\begin{equation}
G_I\>:=\>\left\{\sigma\in G\>\left|\;\> \frac{\sigma b -b}{b}\in I
\mbox{ \ for all }b\in L^\times\right.\right\}
\end{equation}
(see \cite[\S12]{ZS2}). Note that $G_{\cM_L}$ is the ramification group of 
$(L|K,v)$. For every ideal $I$ of ${\cal O}_L\,$, $G_I$ is a
normal subgroup of $G$ (\cite[(d) on p.~79]{ZS2}). The function
\begin{equation}                            \label{eq12}
\varphi:\; I\>\mapsto\>G_I
\end{equation}
preserves $\subseteq$, that is, if $I\subseteq J$, then $G_I\subseteq G_J\,$.
As ${\cal O}_L$ is a valuation ring, the set of its ideals is linearly
ordered by inclusion. This shows that also the higher ramification
groups are linearly ordered by inclusion. Note that in general, $\varphi$ will neither 
be injective nor surjective as a function to the set of normal subgroups of~$G$.

Using the function (\ref{S->IS}), the higher ramification groups can be represented as
\[
G_\Sigma\>:=\>G_{I_\Sigma}\>=\>
\left\{\sigma\in G\>\left|\;\> v\,\frac{\sigma b-b}{b}\in \Sigma\cup\{\infty\}
\mbox{ \ for all }b\in L^\times\right.\right\} \>,
\]
where $\Sigma$ runs through all final segments of $vL^{> 0}$.

Like the function (\ref{eq12}), also the function $\Sigma\mapsto G_\Sigma$ is in 
general not injective. We call $\Sigma$ a \bfind{ramification jump} if
\[
\Sigma'\subsetneq\Sigma\>\Rightarrow\> G_{\Sigma'}\subsetneq G_\Sigma
\]
for all final segments $\Sigma'$ of $vL^{>0}$.
If $\Sigma$ is a ramification jump, then $I_\Sigma$ is called a \bfind{ramification
ideal}. It follows from the definition that an ideal $I$ of ${\cal O}_L$ is a 
ramification ideal if and only if for some subgroup $G'$ of $G$, $I$ is the smallest 
ideal such that $G'=G_I$ (cf.\ \cite[\S 21]{En}).

\parm
In this paper we are particularly interested in the case where $\cE=(L|K,v)$ is a Galois
extension of prime degree $p$. Then $G=\Gal(L|K)$ is a cyclic group of order $p$ and 
thus has only one proper subgroup, namely $\{\mbox{\rm id}\}$, and this subgroup is 
equal to $G_\Sigma$ for $\Sigma=\emptyset$. If in this case $G$ itself is the
ramification group of the extension, then there must be a unique ramification jump
$\Sigma_\cE$, and we will call $I_\cE=I_{\Sigma_\cE}$ {\it the} ramification ideal of 
$(L|K,v)$. As we will show in the next section, ramification jump and ramification 
ideal carry important information about the extension $(L|K,v)$.

\mn
%
%
%------------------------------------------------------------
%
\subsection{Galois defect extensions of prime degree} \label{sectGaldefdegp}
\mbox{ }\sn
The following is part of Theorem~3.5 of \cite{KuRz}:
\begin{theorem}                                        \label{dist_galois_p}
Take a Galois defect extension $\cE=(L|K,v)$ of prime degree with Galois group $G$.
Then $G$ is the ramification group of $\cE$.
The set $\Sigma_\sigma$ does not depend on the choice of
the generator $\sigma$ of $G$. Writing $\Sigma_{\cE}$ for $\Sigma_\sigma\,$, we have 
that the final segment $\Sigma_{\cE}$ of $vK^{>0}$ 
%
%is a final segment of $vK^{>0}$ and satisfies
%\[
%\Sigma_{\cE} \>=\> \Sigma_-(G) \>=\> \Sigma_+(\{\mbox{\rm id}\})\>,
%\]
%showing that $\Sigma_{\cE}$ 
is the unique ramification jump of the extension $\cE$.
Further, 
%
%the ramification ideal $I_\cE=I_{\Sigma_{\cE}}=I_-(G)$ is equal to the ideal of 
%$\cO_L$ generated by the set
%\begin{equation}                        \label{genI-G}
%\left\{ \left.\frac{\sigma b-b}{b}\, \right| \, b\in L^{\times} \right\}\>,
%\end{equation}
%for any generator $\sigma$ of $G$; it 
%
$I_\cE=I_{\Sigma_{\cE}}$ is the unique ramification ideal of the extension $\cE$.
\pars
If $(L|K,v)$ is an Artin-Schreier defect extension with any Artin-Schreier generator 
$\vartheta$, then
\begin{equation}                            \label{SigmaAS}
\Sigma_{\cE} \>=\> -v(\vartheta-K)\>.
\end{equation}
If $K$ contains a primitive root of unity and $(L|K,v)$ is a Kummer extension with 
Kummer generator $\eta$ of value $0$, then
\begin{equation}                            \label{SigmaKum}
\Sigma_{\cE} \>=\> \frac{1}{p-1}vp-v(\eta-K)\>.
\end{equation}
\end{theorem}

\begin{corollary}                            \label{SigmaE-}
Let the assumptions be as in the preceding theorem. Then for every $c\in K$,
\[
v(\vartheta-c)\><\> 0\quad\mbox{ and }\quad v(\eta-c)\><\> \frac{1}{p-1}vp\>. 
\]
\end{corollary}
\begin{proof}
Our assertions follow from (\ref{SigmaAS}), (\ref{SigmaKum}), and the fact that 
$\Sigma_{\cE}\subseteq vK^{>0}$ by Theorem~\ref{dist_galois_p}.
\end{proof}

\begin{proposition}                             \label{ind=I^p=I}
Assume that the extension $\cE=(L|K,v)$ satisfies (\ref{as}).
\sn
1) The following assertions are equivalent:
\sn
a) the extension $\cE$ has independent defect,
\sn
%\begin{equation}                                \label{I_E^p=I_E}
b) $\Sigma_\cE\>=\>(p\Sigma_\cE)\fs\,$,
%\end{equation}
\sn
c) $I_\cE^p\>=\>I_\cE\,$. 
\sn
d) $(I_\cE\cap K)^p\>=\>I_\cE\cap K$. 

\sn
If $vK$ is $p$-divisible, then these properties are also equivalent to 
$\Sigma_\cE\>=\>p\Sigma_\cE\,$.

\mn
2) If $\cE$ has independent defect, then the corresponding convex subgroup $H_\cE$ of 
$vL$ does not contain $vp$.
\end{proposition}
\begin{proof}
1): By definition, $\cE$ has independent defect if and only if $\Sigma_\cE=vL^{\geq 0}
\setminus H_\cE$ for some strongly convex subgroup $H_\cE$ of $vL$. By Lemma~\ref{SD}, 
this holds if and only if $\Sigma_\cE=(p\Sigma_\cE)\fs\,$. By part 2) of
Lemma~\ref{lemI^n=I}, this in turn is equivalent to $I_\cE^p=I_\cE$ since 
$I_\cE=I_{\Sigma_\cE}\,$. It is also equivalent to $(I_\cE\cap K)^p\>=\>I_\cE\cap K$
because $vL=vK$ and therefore, $I_\cE\cap K=(a\in K\mid va\in \Sigma_\cE)$.

\pars
The last assertion of part 1) holds since if $vK$ is $p$-divisible, then $p\Sigma_\cE$
is again a final segment, whence $(p\Sigma_\cE)\fs=p\Sigma_\cE$.

\sn
2): If $\chara K=p$, then $vp=\infty$ and the assertion is trivial. In the mixed
characteristic case, where $\cE$ is a Kummer extension and admits a Kummer generator 
$\eta$ of value $0$, we know that $0=v(\eta-0)\in v(\eta-K)$, hence also $0\in -v(\eta
-K)$. It follows from (\ref{SigmaKum}) that $\frac{1}{p-1}vp\in \Sigma_\cE\,$. Therefore,
$vp\notin H_\cE$ since otherwise $\frac{1}{p-1}vp\in H_\cE$ by convexity; but this is a
contradiction.
\end{proof}

\bn
%
%--------------------------------------------------------------------------
%
\section{Generation of immediate unibranched extensions of valuation rings}
In this section we will assume that $(L|K,v)$ is an immediate extension, and in various
cases determine generators for the valuation ring $\cO_L$ as an $\cO_K$-algebra.

\mn
%
%--------------------------------------------------------------------------
%
\subsection{The general case}
\mbox{ }\sn
In this section we consider an immediate but not necessarily algebraic extension
$(K(x)|K,v)$. Then the set $v(x-K)$ has no maximal element. For this and 
the definition of ``pure extension'', see Section~\ref{sectimmdeflext}.

For every $\gamma\in vK$ we choose $t_\gamma\in K$ such that $vt_\gamma=-\gamma$. For 
every $c\in \cO_K$ we know that $v(x-c)\in vK$ since the extension is immediate, so we 
may set $t_c:=t_{v(x-c)}$ and
\[
x_c \>:=\> t_c(x-c)\,\in\, \cO_{K(x)}^\times\>. 
\]

We use the Taylor expansion (\ref{Taylorexp}).
\begin{lemma}                       \label{l1}
Assume that the immediate extension $(K(x)|K,v)$ is pure. Then for every $g(x)\in 
\cO_{K(x)}\cap K[x]$ there is $c\in K$ such that $g(x)\in \cO_K[x_c]$. If in addition 
$K(x)|K$ is algebraic, then
\[
\cO_{K(x)}\>=\>\bigcup_{c\in K} \cO_K[x_c]\>. 
\]
\end{lemma}
\begin{proof}
%Take $b\in \cO_{K(x)}\cap K[x]$ and write $b=g(x)$ with $\deg g<[K(x):K]$. We note
%that $\deg \partial_i g(X)\leq\deg g$ for $i\geq 0$. Since $(K(x)|K,v)$ is pure, there
%is $\alpha\in v(x-K)$ such that all values $v\partial_i g(c)$ are constant for all 
%$c\in K$ with $v(x-c)\geq \alpha$. As in the proof of Theorem 2 of \cite{Ka} it is shown 
%
From Lemma~\ref{Kap} we infer that whenever for $c\in K$ the value $v(x-c)$ is large 
enough, then the values $v\partial_i g(c) (x-c)^i$, $0\leq i\leq\deg g$, are pairwise
distinct. It follows that  
\[
vg(x)\>=\>\min_i v\partial_i g(c) (x-c)^i\>=\>\min_i v\partial_i g(c)t_c^{-i} x_c^i\>.
\]
Since $g(x)\in\cO_{K(x)}$ and $vx_c=0$, we find that $\partial_i g(c)t_c^{-i}\in\cO_K$ 
for all $i$ and therefore, $g(x)\in \cO_K[x_c]$. 

If $K(x)|K$ is algebraic, then $K(x)=K[x]$ and the last assertion of our lemma follows 
from what we have already proved.
\end{proof}

\begin{lemma}                       \label{l2}
Take $c_1, c_2\in K$. If $v(x-c_1)=v(x-c_2)$, then $\cO_K[x_{c_1}]=\cO_K[x_{c_2}]$.
If $v(x-c_1)<v(x-c_2)$, then $\cO_K[x_{c_1}]\subseteq \cO_K[x_{c_2}]$, and if in addition
$x_{c_1}$ is integral over $K$, then $\cO_K[x_{c_1}]\subsetneq \cO_K[x_{c_2}]$.
\end{lemma}
\begin{proof}
Assume that $v(x-c_1)=v(x-c_2)$. Then $t_{c_1}=t_{c_2}$ and
\[
v(c_1-c_2)\>\geq\>\min\{v(x-c_1),v(x-c_2)\}\>=\> -vt_{c_1}\>, 
\]
whence $t_{c_1}(c_1-c_2)\in\cO_K\,$. It follows that 
\[
\cO_K[x_{c_1}]\>=\>\cO_K[x_{c_1}+t_{c_1}(c_1-c_2)]\>=\>\cO_K[t_{c_1}(x-c_2)]
\>=\>\cO_K[t_{c_2}(x-c_2)]\>=\>\cO_K[x_{c_2}]\>.
\]
Now assume that $v(x-c_1)<v(x-c_2)$. Then $v(c_1-c_2)=\min\{v(x-c_1),v(x-c_2)\}=
v(x-c_1)=-vt_{c_1}<-vt_{c_2}$, whence $t_{c_1}/t_{c_2}\in \cM_K$ and 
$t_{c_1}(c_1-c_2)\in\cO_K^\times$. This shows that
\begin{equation}                         \label{xc2}  
x_{c_1} \>=\> \frac{t_{c_1}}{t_{c_2}}\cdot t_{c_2}(x-c_2)\,+\, t_{c_1}(c_2-c_1)
\,\in\,\cO_K[x_{c_2}]\>,                        
\end{equation}
whence $\cO_K[x_{c_1}]\subseteq \cO_K[x_{c_2}]$. 
\pars
Suppose that $x_{c_2}\in \cO_K[x_{c_1}]$ and $x_{c_1}$ is integral over $K$. 
Then $x_{c_2}$ can be written in the form
\[
x_{c_2}\>=\>a_0+a_1 x_{c_1}+a_2 x_{c_1}^2+\ldots+a_{p-1} x_{c_1}^{n-1}
\]
with $n=[K(x):K]$ and coefficients in $\cO_K\,$. Since the powers of $x_{c_1}$ 
appearing in this expression form a basis of $K(x)|K$, this representation
of $x_{c_2}$ is unique even if one allows the coefficients to be in $K$.
However, according to (\ref{xc2}) we can write $x_{c_1}=ax_{c_2}+b$,
where $a\in \cM_K$ and $b\in \cO_K^\times$. This yields the unique representation
\[
x_{c_2}\>=\> a^{-1}(x_{c_1}-b)\>,
\]
which does not have its coefficients in $\cO_K\,$. This contradiction shows that 
$\cO_K[x_{c_1}]\ne \cO_K[x_{c_2}]$. 
\end{proof}

From Lemmas~\ref{l1} and~\ref{l2}, we obtain:
\begin{proposition}                       \label{p1}
Assume that the immediate algebraic extension
$(K(x)|K,v)$ is pure. Then the rings $\cO_K[x_c]$, $c\in \cO_K\,$, form a chain under
inclusion whose union is $\cO_{K(x)}\,$. 
\end{proposition}

\mn
%
%--------------------------------------------------------------------------
%
\subsection{The case of immediate Artin-Schreier and Kummer extensions}\label{iASKext}
\mbox{ }\sn
An extension of prime degree is a defect extension if and only 
if it is immediate and unibranched. 

\begin{theorem}                             \label{O_LforASK}
1) Assume that $(L|K,v)$ is an Artin-Schreier defect extension with 
Artin-Schreier generator $\vartheta$. 
%Then $v(\vartheta-c)<0$ for all $c\in K$.
The rings $\cO_K[\vartheta_c]$, where $\vartheta_c=t_{v(\vartheta-c)}(\vartheta-c)$ and 
$c$ runs through all elements of $K$, form a chain under inclusion whose union is 
$\cO_L\,$. 
%Note that all elements of $\vartheta-K$ are Artin-Schreier generators of the extension. 
The same holds for the rings $\cO_K[t_{v\vartheta} \vartheta]$ when $\vartheta$ runs 
through all Artin-Schreier generators of the extension. 

For $c_1,c_2\in K$, 
\begin{equation}                         \label{thchain}
\cO_K[\vartheta_{c_1}]\>\subsetneq\> \cO_K[\vartheta_{c_2}]
\>\Leftrightarrow\> v(\vartheta-c_1)< v(\vartheta-c_2)\>.
\end{equation}

\sn
2) Assume that $(L|K,v)$ is a Kummer defect extension of degree $p$ and that $\eta$  
is a Kummer generator which is a $1$-unit. 
%Then $v(\eta-c)<\frac{1}{p-1}vp$ for all $c\in K$. 
The rings $\cO_K[\eta_c]$, where $\eta_c=t_{v(\eta-c)}(\eta-c)$
and $c$ runs through all 1-units of $K$, form a chain under inclusion whose union is 
$\cO_L\,$. If $c_1,c_2\in K$ are 1-units, then
\begin{equation}                            \label{etachain}
\cO_K[\eta_{c_1}] \>\subsetneq\> \cO_K[\eta_{c_2}]
\>\Leftrightarrow\> v(\eta-c_1)<v(\eta-c_2)\>.
\end{equation} 
\end{theorem}
\begin{proof}
From Lemma~\ref{lp} and Proposition~\ref{p1} we know that $\cO_L$ is the union of
the chain $(\cO_K[\vartheta_c])_{c\in K}$ in the Artin-Schreier extension case, and of 
the chain $(\cO_K[\eta_c])_{c\in K}$ in the Kummer extension case. Further, 
(\ref{thchain}) and (\ref{etachain}) follow from Lemma~\ref{l2}.

It remains to prove the additional assertion for the Artin-Schreier extension case.
By \cite[Lemma 2.26]{Ku30}, $\vartheta'\in L$ is another Artin-Schreier generator of 
$L|K$ if and only if  $\vartheta' =i\vartheta -c$ for some $i\in\F_p^\times$ and 
$c\in K$. In particular, $\vartheta-c$ is an Artin-Schreier generator for all $c\in K$.
Moreover, if $\vartheta' =i\vartheta -c$, then $v\vartheta'=v(i\vartheta -c)=
v(\vartheta-i^{-1}c)$ and therefore, $t_{v\vartheta'}=t_{i^{-1}c}$ and 
\[
\cO_K[t_{v\vartheta'} \vartheta'] \>=\> \cO_K[t_{v\vartheta'} i^{-1}\vartheta'] \>=\>
\cO_K[t_{v(\vartheta-i^{-1}c)} (\vartheta-i^{-1}c)]\>=\>\cO_K[\vartheta_{i^{-1}c}]\>.
\]
This shows that the ring $\cO_K[t_{v\vartheta'} \vartheta']$ is already a member in the 
chain constructed above and the union is $\cO_L\,$.
\end{proof}

\parb
Let us consider the {\bf Artin-Schreier case} a bit further, keeping the assumptions of 
part 1) of Theorem~\ref{O_LforASK}. We set $t_c:=t_{v(\vartheta-c)}$. Take any elements
$c_1,c_2\in K$. It follows from (\ref{xc2}) that 
\begin{equation}                         \label{thetac1c2}
\vartheta_{c_1}\>=\>\frac{t_{c_1}}{t_{c_2}}\vartheta_{c_2}+t_{c_1}(c_2-c_1)\>.
\end{equation}

\pars
If $v(\vartheta-c_1)\leq v(\vartheta-c_2)$, then
%
%Then $vt_{c_1}\geq vt_{c_2}$ and
%\[
%v(c_1-c_2)\>\geq\>\min\{v(\vartheta-c_1),v(\vartheta-c_2)\}\>=\>v(\vartheta-c_1) 
%\>=\> -vt_{c_1}\>,
%\]
%whence 
%
$t_{c_1}/t_{c_2}\in\cO_K$ and $t_{c_1}(c_2-c_1)\in\cO_K\,$, as shown in the proof of 
Lemma~\ref{l2}.

\parm
Now we turn to the {\bf Kummer case}, keeping the assumptions of part 2) of
Theorem~\ref{O_LforASK}. We choose a primitive $p$-th root $\zeta_p$ of unity and set
\[
\chi\>:=\> (\zeta_p - 1)^{-1}\>. 
\]
According to equation~(\ref{vz-1}), the element $\chi$ satisfies:
\begin{equation}                         \label{vchi} 
v\chi\>=\>-\frac{1}{p-1}vp\>.
\end{equation}
Hence by Corollary~\ref{SigmaE-},
\[
v\chi(\eta-c)\><\>0 
\]
for all $c\in K$. In the following we will use the abbreviations 
\[
T_c\>:=\> t_{v\chi(\eta-c)} \quad \mbox{ and }\quad \eta_{<c>}\>:=\>T_c\chi(\eta-c) 
\]
in order to simplify our formulas; we observe that  
\begin{equation}                           \label{vT_c}
vT_c\>=\> \frac{1}{p-1}vp - v(\eta-c)\> >\>0\>.
\end{equation}
\begin{corollary}                             \label{CorKumIm}
Assume that $(L|K,v)$ is a Kummer defect extension of degree $p$. Then for every 
Kummer generator $\eta$ which is a $1$-unit, the rings $\cO_K[\eta_{<c>}]$,
where $c$ runs through all 1-units of $K$, form a 
chain under inclusion whose union is~$\cO_L\,$. 

We have that
\begin{equation}                        \label{eqrings'}
\cO_K[\eta_{<c>}] \>=\> \cO_K[\eta_c]\>.  
\end{equation}
Hence if $c_1$ and $c_2$ are 1-units, then
\begin{equation}                            \label{etachain'}
\cO_K[\eta_{<c_1>}] \>\subsetneq\> \cO_K[\eta_{<c_2>}]
\>\Leftrightarrow\> v(\eta-c_1)<v(\eta-c_2)\>.
\end{equation} 
\end{corollary}
\begin{proof}
It suffices to prove (\ref{eqrings'}). We compute:
\[
v\frac{T_c}{\chi^{-1}t_c}\>=\>
v\frac{t_{v\chi(\eta-c)}}{t_{v\chi}t_{v(\eta-c)}}\>=\>0\>. 
\]
Hence,
%As also $t_{-v\chi} \chi^{-1}\in \cO_K^\times$,
%\begin{eqnarray*}
\[
\cO_K[\eta_{<c>}]\>=\>\cO_K\left[\frac{T_c}{\chi^{-1}t_c}\,\chi^{-1}t_c\,\chi(\eta-c)
\right]\>=\> \cO_K[t_c(\eta -c)] \>=\> \cO_K[\eta_c]\>.
\]
%\end{eqnarray*}
Now the assertion of our corollary follows from part 2) of Theorem~\ref{O_LforASK}.
\end{proof}

Take any elements $c_1,c_2\in K$. Similarly as in the Artin-Schreier case, one derives 
the equation
\begin{equation}                            \label{etac1c2}
\eta_{<c_1>} \>=\>\frac{T_{c_1}}{T_{c_2}}\eta_{<c_2>}+T_{c_1}\chi(c_2-c_1)\>.
\end{equation}

\pars
Assume that $v(\eta-c_1)\leq v(\eta-c_2)$. Then $vT_{c_1}\geq vT_{c_2}$ and
\[
v(c_2-c_1)\>\geq\>\min\{v(\eta-c_1),v(\eta-c_2)\}\>=\>v(\eta-c_1) 
\>=\> -vT_{c_1}\chi\>,
\]
whence $T_{c_1}/T_{c_2}\in\cO_K$ and $T_{c_1}\chi(c_2-c_1)\in\cO_K\,$.

\mn
%
%--------------------------------------------------------------------------
%
\subsection{Values of derivatives of minimal polynomials}
\mbox{ }\sn
In the case of Artin-Schreier extensions and Kummer extensions $(K(x)|K,v)$ of prime
degree we have sufficient information about 
the minimal polynomials $f$ of the various generators $x$ we have worked with in the
previous sections, or equivalently, about their conjugates, to work out the values 
$vf'(x)$. In order to do this, we can either compute $f'$ explicitly, or we can use the 
formula
\begin{equation}                              \label{vf'}
f'(x)\>=\>\prod_{\sigma\in G\setminus\{\mbox{\tiny\rm id}\}} (x-\sigma x)\>,
\end{equation}
where $G$ is the Galois group of $K(x)|K$.

We keep the notations from the previous sections.

\mn
%
%--------------------------------------------------------------------------
%
\subsubsection{Artin-Schreier extensions}
\mbox{ }\sn
Take an Artin-Schreier polynomial $f$ with $\vartheta$ as its root. Then $f(X)=X^p-X-
\vartheta^p+\vartheta$ and $f'(X)=-1$, whence 
\begin{equation}                 \label{vf'AS}
f'(\vartheta)=-1\>.
\end{equation} 
This is also obtained from (\ref{vf'}) since $\{\vartheta-\sigma\vartheta\mid\sigma\in
G\setminus\{\mbox{id}\}\}=\F_p^\times$ and the product of all elements of $\F_p^\times$
is $-1$.

\pars
For $t\in K^\times$, denote by $f_t$ the minimal polynomial of $t\vartheta$. Then
\begin{equation}                                 \label{vg'}
f_t'(t\vartheta)\>=\>\prod_{\sigma\in G\setminus\{\mbox{\tiny\rm id}\}} 
(t\vartheta-\sigma t\vartheta)\>=\>t^{p-1}f'(\vartheta)\>=\>-t^{p-1}\>. 
\end{equation}

\begin{lemma}                             \label{vderimmAS}
Take an Artin-Schreier defect extension $\cE=(L|K,v)$ of prime degree 
$p$ with Artin-Schreier generator $\vartheta$. Denote by $g_c$ the minimal polynomial of 
$\vartheta_c\,$. Then the $\cO_L$-ideal $(g_c'(\vartheta_c)\mid c\in K)$ is equal to
\[
I^{p-1} 
\]
where
\begin{equation}                    \label{IEforASdef}
I\>=\>(t_c\mid c\in K)\>=\>(a\in L\mid va\in -v(\vartheta-K))\>=\>I_\cE
\end{equation}
is the ramification ideal of the extension $(L|K,v)$.
\end{lemma}
\begin{proof}
Applying (\ref{vg'}) to $\vartheta_c=t_c(\vartheta-c)$ and keeping in mind that 
$\vartheta-c$ is an Artin-Schreier generator, we obtain:
\begin{equation}                       \label{ASDLDer}
g'_c(\vartheta_c)\>=\>-t_c^{p-1} \>. 
\end{equation}
This shows that $(g_c'(\vartheta_c)\mid c\in K)=(t_c\mid c\in K)^{p-1}=I^{p-1}\,$.

The second equality in (\ref{IEforASdef}) holds since $vt_c=-v(\vartheta-c)$ by 
definition. The third equality follows from equation (\ref{SigmaAS}) of
Theorem~\ref{dist_galois_p}. 
\end{proof}

\mn
%
%--------------------------------------------------------------------------
%
\subsubsection{Kummer extensions}
\mbox{ }\sn
Take $f(X)=X^p-\eta^p$. Then $f'(X)=pX^{p-1}$, whence 
\begin{equation}                    \label{vf'eta}
f'(\eta)\>=\>p\eta^{p-1}\>.
\end{equation}
This is also obtained from equations (\ref{vf'}) and (\ref{prod1-z}) using that 
$\{\sigma\eta\mid \sigma\in G\}=\{\zeta_p^i\eta\mid i\in\{0,\ldots,p-1\}\}$. 

\begin{lemma}                             \label{vderimmK}
Take a Kummer defect extension $\cE=(L|K,v)$ of prime degree $p$ with Kummer 
generator $\eta$ which is a $1$-unit. Denote by $f_c$ the minimal polynomial of 
$\eta_{<c>}$ over 
$K$. Then the $\cO_L$-ideal $(f'_c(\eta_{<c>})\mid c\in K\mbox{ a 1-unit})$ is
\[
I^{p-1} 
\]
where
\begin{equation}                    \label{IEforKumdef}
I\>=\>(T_c\mid c\in K)\>=\>\left(a\in L\mid va\in \frac{1}{p-1}vp-v(\eta-K)\right)
\>=\>I_\cE
\end{equation}
is the ramification ideal of the extension $(L|K,v)$. 
\end{lemma}
\begin{proof}
We use (\ref{vf'}) to compute:
\[
f'_c(\eta_{<c>}) \>=\> \prod_{\sigma\in G} (\eta_{<c>}-\sigma\eta_{<c>})
\>=\>\prod_{\sigma\in G} T_c \chi(\eta -\sigma\eta ) \>=\> 
(T_c\chi\eta)^{p-1}\prod_{i=1}^{p-1} (1-\zeta_p^i)\>,
\]
whence by equations (\ref{prod1-z}) and (\ref{vchi}), and the fact that $v\eta =0$,
\begin{equation}                          \label{f'_c}
f'_c(\eta_{<c>}) \>=\> p(T_c\chi\eta)^{p-1}\>=\> T_c^{p-1} u\>
\end{equation}
with $u$ a unit in $\cO_L\,$.
%In view of (\ref{vchi}) and $v\eta =0$, we find that
%\begin{equation}                          \label{vf'_c}
%v f'_c(\eta_{<c>})\>=\>(p-1)vT_c\>. %\>=\>(p-1)\left(\frac{1}{p-1}vp-v(\eta-c)\right)\>.
%\end{equation}
This shows that 
\[
(f'_c(\eta_{<c>})\mid c\in K\mbox{ a 1-unit})=(T_c\mid c\in K)^{p-1}=I^{p-1}\>.
\]

The second equality in (\ref{IEforKumdef}) follows from (\ref{vT_c}). The
assertion that $I= I_{\cE}$ in the Kummer case follows from equation~(\ref{SigmaKum})
of Theorem~\ref{dist_galois_p}. 
\end{proof}

\bn
%
%--------------------------------------------------------------------------
%
\section{Kähler differentials and their annihilators for algebraic field extensions}                                 
\label{sectKd}
%
%\mn
%
%--------------------------------------------------------------------------
%
\subsection{Proof of Theorem~\ref{LimProp}}
%Basic results on the representation of Kähler differentials} 
\label{sect4.1}
\mbox{ }\sn
Let $L|K$ be an algebraic field extension. Let $A\subseteq K$ be a 
normal domain whose quotient field is $K$. Assume that $z\in L$ is integral over $A$ 
and let $f(X)$ be the minimal polynomial of $z$ over $K$. Then $f(X)\in A[X]$ 
(see \cite[Theorem 4, page 260]{ZSI}). Since $f(X)$ is monic, $(f(X)K[X])\cap R[X]=
f(X)R[X]$, so $A[z]\cong A[X]/(f(X))$. Thus, 
\begin{equation}                                  \label{Kahlercomp}
\Omega_{A[z]|A}\>\cong\> [A[X]/(f(X),f'(X))]dX \>\cong\> [A[z]/(f'(z))]dX
\end{equation}
by \cite[Example 26.J, page 189]{Mat} and \cite[Theorem 58, page 187]{Mat}. There is a
canonical derivation $d_{A[z]|A}:A[z]\rightarrow \Omega_{A[z]|A}$ defined by $g(z)\mapsto 
[g'(z)]dX$ for $g(X)\in A[X]$, where $[g'(z)]$ is the class of $g'(z)$ in $A[z]/(f'(z))$.

\parm
We will now proceed under the assumptions of Theorem~\ref{LimProp}.
Before proving the theorem, we derive a formula that we will need in its proof.
Let $n=[L:K]$. Then for all $\alpha\in S$, $n=\deg h_{\alpha}(X)$ where 
$h_{\alpha}(X)$ is the minimal polynomial of $b_{\alpha}$ over $K$. Suppose 
that $\alpha<\beta$. Then
\[
\left(\frac{a_{\beta}}{a_{\alpha}}\right)^nh_{\alpha}\left(\frac{a_{\alpha}}{a_{\beta}}
X+c_{\alpha,\beta}\right)
\]
is a monic polynomial of degree $n$ in $K[X]$ which has $b_{\beta}$ as a root, so 
it is the minimal polynomial $h_{\beta}(X)$ of $b_{\beta}$. By the chain rule, 
\[
\frac{dh_{\beta}}{d_{X_{\beta}}}(b_{\beta})=\left(\frac{a_{\beta}}{a_{\alpha}}\right)^n
\frac{dh_{\alpha}}{dX}(b_{\alpha})\left(\frac{a_{\alpha}}{a_{\beta}}\right),
\]
and so
\begin{equation}                            \label{eqCR}
h_{\alpha}'(b_{\alpha})=\left(\frac{a_{\alpha}}{a_{\beta}}\right)^{n-1}
h_{\beta}'(b_{\beta}).
\end{equation}

\sn
Proof of Theorem~\ref{LimProp}:\n
From the natural $A$-algebra isomorphisms $A[b_{\alpha}]\cong A[X_{\alpha}]/(h_{\alpha}
(X_{\alpha}))$, where $h_{\alpha}(X_{\alpha})$ is the minimal polynomial of $b_{\alpha}$
over $K$, we have that for $\alpha\le\beta$, the natural inclusion of $A$-algebras
$A[b_{\alpha}]\subset  A[b_{\beta}]$ is induced by sending $b_\alpha$ to 
$\frac{a_{\alpha}}{a_{\beta}}b_{\beta}+c_{\alpha,\beta}$.
%the substitution 
%$$
%X_{\alpha}=\frac{a_{\alpha}}{a_{\beta}}X_{\beta}+c_{\alpha,\beta}\>.
%$$
%Thus
%$d_{A[X_{\beta}]|A}(X_{\alpha})=\frac{a_{\alpha}}{a_{\beta}}dX_{\beta}$. 
 
By
(\ref{Kahlercomp}), for $\alpha\in S$ we have a natural
isomorphism of $A[b_{\alpha}]$-modules 
\[
\Omega_{A[b_{\alpha}]|A}
\cong \left[A[b_{\alpha}]/(h_{\alpha}'(b_{\alpha}))\right]dX_{\alpha}\>,
\]
and so 
\[
\Omega_{A[b_{\alpha}]|A}\otimes_{A[b_{\alpha}]} B
\cong \left[B/(h_{\alpha}'(b_{\alpha}))\right]dX_{\alpha}\>.
\]
For $\alpha<\beta$, by the universal property of derivations (Proposition on page 182 \cite{Mat}), there is a unique  $B$-module homomorphism
\[
\lambda_{\alpha,\beta}:\Omega_{A[b_{\alpha}]|A}\otimes_{A[b_{\alpha}]}B\rightarrow
\Omega_{A[b_{\beta}]|A}\otimes_{A[b_{\beta}]}B
\]
such that there is a commutative diagram of $B$-module homomorphisms
\[
\begin{array}{ccc}
\Omega_{A[b_{\alpha}]|A}\otimes_{A[b_{\alpha}]}B&\stackrel{\lambda_{\alpha,\beta}}
{\rightarrow}&\Omega_{A[b_{\beta}]|A}\otimes_{A[b_{\beta}]}B\\
\uparrow&&\uparrow \\
A[b_{\alpha}]\otimes_{A[b_{\alpha}]}B&\rightarrow &A[b_{\beta}]\otimes_{A[b_{\beta}]}B
\end{array}
\]  
where the vertical arrows are the respective maps
$$
d_{A[b_{\alpha}]/A}\otimes 1:A[b_{\alpha}]\otimes_{A[b_{\alpha}]}B\rightarrow
\Omega_{A[b_{\alpha}]|A}\otimes_{A[b_{\alpha}]}B
$$ 
and
$$
d_{A[b_{\beta}]/A}\otimes 1:A[b_{\beta}]\otimes_{A[b_{\beta}]}B\rightarrow
\Omega_{A[b_{\beta}]|A}\otimes_{A[b_{\beta}]}B.
$$   
Suppose that $z\in A[b_{\alpha}]$. Then there exists $g(X_{\alpha})\in A[X_{\alpha}]$ 
such that $z=g(b_{\alpha})$, hence $d_{A[b_{\alpha}]/A}(z)=\frac{dg}{dX_{\alpha}}
(b_{\alpha})dX_{\alpha}$. We have that $z=\left(g(\frac{a_{\alpha}}{a_{\beta}}X_{\beta}
+c_{\alpha,\beta})\right)(b_{\beta})$  so that by the chain rule, 
$$
d_{A[b_{\beta}]/A}(z)\>=\>\left(\frac{d}{dX_{\beta}}g\left(\frac{a_{\alpha}}{a_{\beta}}
X_{\beta}+c_{\alpha,\beta}\right)\right)(b_{\beta})dX_{\beta}
\>=\>\frac{dg}{dX_{\alpha}}(b_{\alpha})\frac{a_{\alpha}}{a_{\beta}}dX_{\beta}.
$$  
Thus $\lambda_{\alpha,\beta}$ is the $B$-module homomorphism defined by mapping 
$dX_{\alpha}$ to $\frac{a_{\alpha}}{a_{\beta}}dX_{\beta}\,$. 

In order to compute the direct limit of the directed system of $B$-module homomorphisms 
$\lambda_{\alpha,\beta}:\Omega_{A[b_{\alpha}]|A}\otimes_{A[b_{\alpha}]}B\rightarrow
\Omega_{A[b_{\beta}]|A}$ for $\alpha<\beta$, we will introduce an equivalent directed
system, which is a little simpler. 
For $\alpha\in S$, let $M_{\alpha}$ be the $B$-module $M_{\alpha}=B/(h'_{\alpha}
(b_{\alpha}))$. The $M_{\alpha}$ are a family of $B$-modules. We have isomorphisms of 
$B$-modules 
$\tau_{\alpha}:\Omega_{A[b_{\alpha}]|A}\otimes_{A[b_{\alpha}]}B\rightarrow M_{\alpha}$
defined by mapping $dX_{\alpha}$ to~$1$.

For all $\alpha,\beta\in S$ with $\alpha<\beta$, we have commutative diagrams of 
$B$-module homomorphisms
$$
\begin{array}{ccc}
\Omega_{A[b_{\alpha}]|A}\otimes_{A[b_{\alpha}]}B&\stackrel{\lambda_{\alpha,\beta}}
{\rightarrow}&\Omega_{A[b_{\beta}]|A}\otimes_{A[b_{\beta}]}B\\
\downarrow\tau_{\alpha}&&\downarrow \tau_{\beta}\\
M_{\alpha}&\stackrel{\tau_{\beta}\lambda_{\alpha,\beta}\tau_{\alpha}^{-1}}{\rightarrow}
&M_{\beta}.
\end{array}
$$
and we see that $\tau_{\beta}\lambda_{\alpha,\beta}\tau_{\alpha}^{-1}$ is just
multiplication by $\frac{a_{\alpha}}{a_{\beta}}$.
Thus the directed systems  
\[
\Omega_{A[b_{\alpha}]|A}\otimes_{A[b_{\alpha}]}B\stackrel{\lambda_{\alpha,\beta}}
\rightarrow \Omega_{A[b_{\beta}]|A}\otimes_{A[b_{\beta}]}B \,\mbox{ for }\, \alpha<\beta
\quad\mbox{ and }\quad
M_{\alpha}\stackrel{\frac{a_{\alpha}}{a_{\beta}}}\rightarrow M_{\beta} 
\,\mbox{ for }\, \alpha<\beta 
\]
are equivalent.

We have isomorphisms of $B$-modules
$$
\Omega_{B|A}\cong \lim_{\rightarrow}\left[\Omega_{A[b_{\alpha}]|A}
\otimes_{A[b_{\alpha}]}B\right]\cong \lim_{\rightarrow}M_{\alpha}
$$
by \cite[Corollary 16.7, page 394]{Eis}, where the direct limits are over $\alpha\in S$.

Write $M_{\alpha}=R_{\alpha}/T_{\alpha}$ where $R_{\alpha}$ is the directed system of
$B$-modules $R_{\alpha}=B$ for $\alpha\in S$, with $B$-module homomorphisms
$\frac{a_{\alpha}}{a_{\beta}}:R_{\alpha}\rightarrow R_{\beta}$ for $\alpha<\beta$, and 
$T_{\alpha}$ is the directed system of $B$-modules $T_{\alpha}=
h_{\alpha}'(b_{\alpha})R_{\alpha}$ for $\alpha\in S$, with $B$-module homomorphisms
$\frac{a_{\alpha}}{a_{\beta}}:T_{\alpha}\rightarrow T_{\beta}$ for $\alpha\le\beta$.
We have short exact sequences of $B$-modules 
$$
0\rightarrow T_{\alpha}\rightarrow R_{\alpha}\rightarrow M_{\alpha}\rightarrow 0
$$
which are compatible with multiplication by $\frac{a_{\alpha}}{a_{\beta}}$ for $\alpha<
\beta$. Thus
$$
\lim_{\rightarrow}M_{\gamma}\cong \lim_{\rightarrow}R_{\gamma}/
\lim_{\rightarrow}T_{\gamma}
$$
by \cite[Theorem 2.18]{Rot}.

We now compute $\displaystyle\lim_{\rightarrow} R_{\gamma}\,$. For $\alpha\in S$, we 
have $B$-module homomorphisms $R_{\alpha}\stackrel{a_{\alpha}}{\rightarrow} B$ which 
give commutative diagrams
$$
\begin{array}{lcr}
R_{\alpha}&\stackrel{\frac{a_{\alpha}}{a_{\beta}}}{\rightarrow}& R_{\beta}\\
a_{\alpha}\searrow&&\swarrow a_{\beta}\\
&B&
\end{array}
$$
for $\alpha<\beta$. By the universal property of direct limits (\cite[Proposition 11.1]
{AG} or \cite[Exercise 18, page 33]{AM}) there exists a unique $B$-module homomorphism 
$\Psi:\displaystyle\lim_{\rightarrow} R_{\gamma}\rightarrow B$ giving  commutative 
diagrams of $B$-module homomorphisms

$$
\begin{array}{rl}
R_{\alpha}&\\
\frac{a_{\alpha}}{a_{\beta}}\downarrow&\searrow\phi_{\alpha}\\
R_{\beta}&\stackrel{\phi_{\beta}}{\rightarrow}\displaystyle\>\lim_{\rightarrow}R_{\gamma}\\
a_{\beta}\downarrow&\swarrow\Psi\\
B&\\
\end{array}
$$
for $\alpha<\beta$. Here $\phi_{\alpha}:R_{\alpha}\rightarrow \displaystyle
\lim_{\rightarrow} R_{\gamma}$ are the canonical maps of the direct product.

As we will now show, it follows from some standard properties of direct limits
(\cite[Proposition 11.3]{AG} 
or \cite[Exercise 15, page 33]{AM}) and some diagram chasing that as $B$-ideals,
$$
\lim_{\rightarrow}R_{\gamma}\>\cong\> (a_{\alpha}\mid \alpha\in S)\>=\>U\>.
$$
 
Suppose that $u\in\displaystyle\lim_{\rightarrow}R_{\gamma}$. Then there exist $\alpha
\in S$ and $b\in R_{\alpha}=B$ such that $u=\phi_{\alpha}(b)$. Thus $\Psi(u)=a_{\alpha}
b\in U$. Given $\alpha\in S$, $\Psi(\phi_{\alpha}(1))=a_{\alpha}$. Thus the image of 
$\Psi$ is $U$. Suppose that $0\ne u\in \displaystyle\lim_{\rightarrow}R_{\gamma}$. 
Then there exist 
$\alpha\in S$ and $0\ne b\in R_{\alpha}$ such that $\phi_{\alpha}(b)=u$ so 
$\Psi(u)=a_{\alpha}b\ne 0$. Thus $\Psi$ is an isomorphism and so
$$
\lim_{\rightarrow}R_{\gamma}\cong U.
$$
Similarly, 
$$
\displaystyle\lim_{\rightarrow}T_{\gamma}\cong (a_{\alpha}h_{\alpha}'(b_{\alpha})\mid
\alpha\in S).
$$
By (\ref{eqCR}), we have that  for $\alpha,\beta<\gamma$,
$$
a_{\alpha}h_{\beta}'(b_{\beta})=\frac{a_{\alpha}}{a_{\gamma}}\left(\frac{a_{\beta}}
{a_{\gamma}}\right)^{n-1}a_{\gamma}h_{\gamma}'(b_{\gamma}).
$$
Thus $(a_{\alpha}h_{\alpha}'(b_{\alpha})\mid \alpha\in S)=UV$ and so  
$\Omega_{B|A}\cong U/UV$.
\qed

%
%--------------------------------------------------------------------------
%

\parb
If we choose any $\gamma\in S$, then we will still have 
\[
\bigcup_{\gamma\leq\alpha\in S}A[b_{\alpha}]\>=\>B 
\]
since $(A[b_\alpha])_{\alpha\in S}$ is an increasing chain.
Hence we can always assume that $S$ has a minimal element $\gamma$. 
 
\pars
From now on we consider the case of $A=\cO_K$ and $B=\cO_L\,$. 
Assume that $\alpha\le\beta$. Then $a_{\beta}\mid a_{\alpha}$, so that 
$\frac{a_\alpha}{a_\beta}\,,\,\left(\frac{a_\alpha}{a_\beta}\right)^{n-1}\in \cO_K\,$. 
From (\ref{b_ab_b}) and (\ref{eqCR}) it thus follows that $b_\alpha\in (b_\beta)$
and $h'_\alpha(b_\alpha)\in (h'_\beta(b_\beta))$. This shows that $\cO_L$-ideals 
$(b_\alpha)_{\alpha\in S}$ and $(h'_\alpha(a_\alpha))_{\alpha\in S}$ form increasing 
chains.

\pars
Using (\ref{eqCR}) with $\beta=\gamma$, we obtain:
\begin{equation}                 \label{altpresV}
V \>=\>(h'_\alpha(b_{\alpha})\mid \alpha\in S)\>=\>
\left(\left(\frac{a_\alpha}{a_\gamma}\right)^{n-1}h'_\gamma(b_{\gamma})
\mid\alpha\in S\right)\>=\>\frac{h'_\gamma(b_{\gamma})}{a_\gamma^{n-1}}U^{n-1}\>.
\end{equation}
This proves (\ref{altpres}).

\mn
%
%--------------------------------------------------------------------------
%
\subsection{The ideal $V$ of differents}      \label{sectVd}
\mbox{ }\sn
Under the assumptions of Theorem~\ref{LimProp}, with $(L|K, v)$ a valued field extension
satisfying condition~(\ref{as}) and $A=\cO_K$ and $B=\cO_L\,$, we will now show that $V$ is 
the $\cO_L$-ideal generated by the differents of the generators $b_\alpha$ of $\cO_L$ over 
$\cO_K\,$. 
Take any generator $b\in\cO_L$ of $L|K$ and let $h$ be its minimal polynomial over $K$. 
Denote the different of $b$ by $\delta(b):=h'(b)$. For $i\geq 1$ we have
\[
b^i-\sigma b^i\>=\> b^i-(\sigma b)^i\>=\>b^i-(b+(\sigma b-b))^i\>=\>
\sum_{j=0}^{i-1} \binom{i}{j} b^j(\sigma b-b)^{i-j}\>.
\]
Since the extension is unibranched, we have $v\sigma b=vb$, whence $v(b-\sigma b)\geq 
vb\geq 0$. Consequently,
\[
v(b^i-\sigma b^i) \>\geq\> v(\sigma b-b)\>=\> v(b-\sigma b)\>.
\]

\pars
Every $b\in \cO_K[b_\alpha]\setminus\cO_K$ is of the form
\[
b\>=\>\sum_{i=0}^{n-1} c_i b_\alpha^i  
\]
with $c_i\in \cO_K\,$.
%, $c_i\ne 0$ for some $i>0$. 
Therefore,
\begin{eqnarray*}
v\delta(b)&=& v\prod_{\mbox{\tiny\rm id}\ne\sigma\in G} (b-\sigma b)\>=\>
\sum_{\mbox{\tiny\rm id}\ne\sigma\in G} v\left(\sum_{i=0}^{n-1} c_i b_\alpha^i
-\sigma \sum_{i=0}^{n-1} c_i (b_\alpha)^i\right)\\
&=& \sum_{\mbox{\tiny\rm id}\ne\sigma\in G} v\sum_{i=1}^{n-1} c_i (b_\alpha^i
-\sigma b_\alpha^i)\>.
\end{eqnarray*}
For $1\leq i\leq n-1$, we have 
\[
v c_i (b_\alpha^i-\sigma b_\alpha^i)\>\geq\>v(b_\alpha^i-\sigma b_\alpha^i)\>\geq\>
v(b_\alpha-\sigma b_\alpha)\>,
\]
showing that 
\[
v\sum_{i=1}^{n-1} c_i (b_\alpha^i-\sigma b_\alpha^i) \>\geq\>v(b_\alpha
-\sigma b_\alpha)\>.
\]
Hence,
\[
v\delta(b) \>\geq\> \sum_{\mbox{\tiny\rm id}\ne\sigma\in G} v(b_\alpha-\sigma b_\alpha)
\>=\> vh'_\alpha(b_\alpha)\>.
\]

We use this to conclude:
\begin{proposition}                             \label{VdO}
Under the above assumptions, we have:
\begin{equation}
V\>=\>(\delta(b) \mid b\in\cO_L\setminus\cO_K)\>.
\end{equation}
\end{proposition}
\begin{proof}
Since $\cO_L$ is the union of the chain of rings 
$\cO_K[b_\alpha]$, $\alpha\in S$, we have
\[
(\delta(b) \mid b\in\cO_L\setminus\cO_K)\>=\>\bigcup_{\alpha\in S} 
(\delta(b) \mid b\in \cO_K[b_\alpha]\setminus\cO_K)\>=\>\bigcup_{\alpha\in S} 
(h'_\alpha (b_\alpha))\>=\> V\>.
\]
\end{proof}

\mn
%
%--------------------------------------------------------------------------
%
\subsection{The case of Artin-Schreier defect extensions} 
%
%\mbox{ }\sn
\begin{theorem}                       \label{OASdef} 
Let $(L|K,v)$ be an Artin-Schreier defect extension with ramification ideal
$I_{\cE}\,$. Then there is an $\cO_L$-module isomorphism
\[
\Omega_{\cO_L/\cO_K}\>\cong\> I_{\cE}/I_{\cE}^p\>.
\]
\end{theorem} 
\begin{proof} 
Let $\vartheta$ be an Artin-Schreier generator of $L|K$.
By Theorem \ref{O_LforASK}, the $\cO_K$-algebras $\cO_K[\vartheta_c]$, 
with $c\in K$, form a chain of subrings of $L$ whose union is $\cO_L$.  For 
$c\in K$, $\vartheta_c=t_c(\vartheta-c)$, where $t_c\in \cO_K$ is such that 
$v\vartheta_c=0$. For  $v(\vartheta-c_1)<v(\vartheta-c_2)$, by (\ref{thetac1c2})
we have that 
\[
\vartheta_{c_1}=\frac{t_{c_1}}{t_{c_2}}\vartheta_{c_2}+t_{c_1}(c_2-c_1)
\]
with $\frac{t_{c_1}}{t_{c_2}}$, $t_{c_1}(c_2-c_1)\in \mathcal O_K$.
We will apply Theorem~\ref{LimProp} with $A=\cO_K\,$,
$B=\cO_L\,$, and $S=v(\vartheta-K)$. For each $\alpha\in S$ we choose $c_\alpha\in K$ 
such that $v(\vartheta-c_\alpha)=\alpha$ and set $b_{\alpha}=\vartheta_{c_\alpha}$. 
We set $a_{\alpha}=t_{c_\alpha}$ and $c_{\alpha,\beta}=t_{c_\alpha}(c_\beta-c_\alpha)$. 
We denote by $h_{\alpha}$ the minimal polynomial of $b_{\alpha}=\vartheta_{c_\alpha}$ 
over $K$. Thus in the notation of Lemma~\ref{vderimmAS}, $h_{\alpha}=g_{c_\alpha}$ so 
that $h'_\alpha(b_{\alpha})=g_{c_\alpha}'(\vartheta_{c_\alpha})=-t_{c_\alpha}^{p-1}$ by
equation (\ref{ASDLDer}). Hence for every $\alpha\in S$,
$h'_\alpha(b_{\alpha})/a_{\alpha}^{p-1}=-1$. Our theorem now follows from
equation~(\ref{altpres}) of Theorem~\ref{LimProp}, together with Lemma~\ref{vderimmAS} 
which shows that the $\cO_L$-ideal $U=(t_{\alpha}\mid\alpha\in S)$ is equal to 
$I_{\cE}\,$.
\end{proof}

\mn
%
%--------------------------------------------------------------------------
%
\subsection{The case of Kummer defect extensions of degree $p$} 
%
%\mbox{ }\sn
\begin{theorem}                            \label{OKumdef} 
Assume that $\chara K=0$ and let $(L|K,v)$ be a Kummer defect extension of degree $p$
with ramification ideal $I_{\cE}\,$. Then there is an $\cO_L$-module isomorphism
\[
\Omega_{\cO_L|\cO_K}\>\cong\> I_{\cE}/I_{\cE}^p\>.
\]
\end{theorem} 
\begin{proof} 
Let $\eta$ be a Kummer generator of $L|K$ which is a $1$-unit. We use the notation from
Section \ref{iASKext}. By Corollary~\ref{CorKumIm}, the $\cO_K$-algebras $\cO_K
[\eta_{<c>}]$ with $c$ a 1-unit in $K$ form a chain of subrings of $L$ whose union is 
$\cO_L$. For $v(\eta-c_1)<v(\eta-c_2)$, by (\ref{etac1c2}) we have that 
\[
\eta_{<c_1>}=\frac{T_{c_1}}{T_{c_2}}\eta_{<c_2>}+T_{c_1}\chi(c_2-c_1)
\]
with $\frac{T_{c_1}}{T_{c_2}}$, $T_{c_1}\chi(c_2-c_1)\in \mathcal O_K$. We will apply
Theorem~\ref{LimProp} with $A=\cO_K$, $B=\cO_L\,$, and $S=v(\eta-K)$. For each 
$\alpha\in S$ we choose $c_\alpha\in K$ such that $v(\eta-c_\alpha)=\alpha$. Let
$b_{\alpha}=\eta_{<c_\alpha>}$, $a_{\alpha}=T_{c_\alpha}$, and $c_{\alpha,\beta}
=T_{c_\alpha}\chi(c_\beta-c_\alpha)$. We denote by $h_{\alpha}$ the minimal polynomial 
of $b_{\alpha}=\eta_{<c_\alpha>}$ over $K$. Thus in the notation of Lemma~\ref{vderimmK},
$h_{\alpha}=f_{c_\alpha}$ so that by equation (\ref{IEforKumdef}), $h'_\alpha(b_{\alpha})
=f_{c_\alpha}'(\eta_{<\alpha>})=T_{c_\alpha}^{p-1} u$ with $u$ a unit in $\cO_L\,$. Hence 
for every $\alpha\in S$, $h'_\alpha(b_{\alpha})/a_{\alpha}^{p-1}=u$. 
Our theorem now follows from equation~(\ref{altpres}) of Theorem~\ref{LimProp}, together 
with Lemma~\ref{vderimmK}, which shows that the $\mathcal O_L$-ideal $U=(T_{\alpha}
\mid\alpha\in S)$ is equal to $I_{\cE}\,$.
\end{proof}

\mn
%
%--------------------------------------------------------------------------
%
\subsection{The valuation ring $\cO(I)$ and the annihilator}    \label{sectann}
\mbox{ }\sn
Take an $\cO_L$-ideal $I$. We define $\cO(I)$ to be the largest valuation ring $\cO$ 
containing $\cO_L$ such that $I$ is an $\cO$-ideal. Further, we define $H(I)$ to be the 
largest convex subgroup $H$ of $vL$ such that $vI+H=vI$. In the terminology of 
\cite{Ku60, Kcuts}, this is called the \bfind{invariance group} of $vI$, denoted by
$\cG(vI)$. We infer from \cite[Theorem~3.6]{Kcuts}:
\begin{proposition}                    \label{O(I)}
For every $\cO_L$-ideal $I$,
\begin{equation}
\cO(I)\>=\> \cO_{v_{H(I)}}\>.
\end{equation}
\end{proposition}

Accordingly, we set $\cM(I):=\cM_{v_{H(I)}}\,$. We will use the following facts:
%The following is Proposition of \cite{Kcuts}:
%
\begin{proposition}                    \label{ann}
Take an $\cO_L$-ideal $U$ and set $V=bU^{n-1}$, where $b\in\cO_L\,$. 
\sn
1) We have $H(U)=H(V)$, $\cO(U)=\cO(V)$ and $\cM(U)=\cM(V)$.
\sn
2) If there is $a\in K$ such that $v_{H(V)}V$ has infimum 
$v_{H(V)}a$ in $v_{H(V)}K$ but $v_{H(V)}V$ does not contain this infimum, then 
\begin{equation}             % \label{annform1}
\ann U/UV\>=\>a\cO(V) \>,
\end{equation}
which properly contains $V$. In all other cases, $\ann U/UV=V$. 

In all cases, $\cM(V)\,\ann U/UV\subseteq V$. 

\sn
3) If $\cM_L$ annihilates $U/UV$, then $U/UV=0$, except if $\cM_L$ is principal and either
$U=V=\cM_L$ (which holds if and only if $n=2$ and $b\in \cO_L^\times$), or $U=\cO_L$ and
$V=\cM_L=(b)$ (which holds if and only if $n=1$ and $b\notin \cO_L^\times$). In these 
exceptional cases, $\cM_L=\ann U/UV$. 
\end{proposition}
\begin{proof}
1) follows from \cite[Lemmas~2.14 and~3.7]{Kcuts}.
2) and 3) follow from \cite[Proposition~3.19]{Kcuts}.
\end{proof}

\pars
In Theorems~\ref{OASdef} and~\ref{OKumdef} we proved that in the case of $\cE=
(L|K,v)$ an Artin-Schreier or a Kummer defect extension of degree $p$ with ramification 
ideal $I_{\cE}$ we have $\Omega_{\cO_L/\cO_K}\cong I_{\cE}/I_{\cE}^p$. Therefore, we will 
now assume that the extension $\cE=(L|K,v)$ satisfies (\ref{as}) and compute the 
corresponding annihilator. We observe:
\begin{lemma}
Assume that the extension $\cE=(L|K,v)$ satisfies (\ref{as}) and has independent defect. 
Then 
\[
H(vI_\cE)\>=\>H_\cE\>. 
\]
\end{lemma}
\begin{proof}
We have $vI_\cE=\Sigma_\cE=\{\alpha\in vL\mid \alpha >H_\cE\}$. In \cite{Kcuts}, this set 
is denoted by $H_\cE^+$, and in \cite[part 5) of Lemma~2.14]{Kcuts} (where we take $\alpha
=0$) it is shown that for each convex subgroup $H$ the invariance group of $H^+$ equals $H$.
\end{proof}

Analogously, we define $H_\cE:=H(vI_\cE)$ in case $\cE$ has dependent defect. Then we 
obtain from Proposition~\ref{O(I)} that
\begin{equation}
\cO(I_\cE)\>=\>\cO_{v_{H_\cE}} 
\end{equation}
holds for all extensions $\cE=(L|K,v)$ satisfying (\ref{as}). Applying 
Proposition~\ref{ann}, we find:
\begin{proposition}                         \label{propann}
Assume that the extension $\cE=(L|K,v)$ satisfies (\ref{as}). 
\sn
1) We have $H(I_\cE)=H(I_\cE^{p-1})=H(I_\cE^p)$, $\cO(I_\cE)=\cO(I_\cE^{p-1})=\cO(I_\cE^p)$ 
and $\cM(I_\cE)=\cM(I_\cE^{p-1})=\cM(I_\cE^p)$.
\sn
2) If there is $a\in K$ such that $v_{H_\cE} I_\cE^{p-1}$ has infimum $v_{H_\cE}a$ in
$v_{H_\cE}K$ but $v_{H_\cE} I_\cE^{p-1}$ does not contain this infimum, then 
\begin{equation}             % \label{annform2}
\ann I_\cE/I_\cE^p\>=\>a\cO(I_\cE) \>,
\end{equation}
which properly contains $I_\cE^{p-1}$. In all other cases, $\ann I_\cE/I_\cE^p=I_\cE^{p-1}$. 

In all cases, $\cM(I_\cE)\,\ann I_\cE/I_\cE^p\subseteq I_\cE^{p-1}$. 
\sn
3) If $\cM_L$ annihilates $\Omega_{\cO_L|\cO_K}$, then $\Omega_{\cO_L|\cO_K}=0$ and 
$\ann\Omega_{\cO_L|\cO_K}=\cO_L\,$.
\end{proposition}
\begin{proof}
Part 1) follows from \cite[Lemmas~2.14 and~3.7]{Kcuts}. Part 2) follows from part 2) of 
Proposition~\ref{ann} where $U=I_\cE$ and $V=I_\cE^{p-1}$, and using that $\cO(I_\cE)=
\cO(I_\cE^{p-1})$. Part 3) follows from Theorems~\ref{OASdef} and~\ref{OKumdef} and part 3) 
of Proposition~\ref{ann} since under our assumptions on $\cE$, $I_\cE^{p-1}$ cannot be principal.
\end{proof}

\bn
%
%--------------------------------------------------------
%
\section{The trace of valuation rings and differents}     \label{secttrace}
%
%
%--------------------------------------------------------
%
\subsection{Defect extensions of prime degree}          %\label{sect}
\mbox{ }\sn
In this section we will consider the trace on an extension $\cE=(L|K,v)$ 
that satisfies (\ref{as}). If $L|K$ is an Artin-Schreier extension, then we write 
$L=K(\vartheta)$ where $\vartheta$ is an Artin-Schreier generator. If $\chara K=0$
and $L|K$ is a Kummer extension, then we write $L=K(\eta)$ where $\eta$ is a Kummer
generator, that is, $\eta^p\in K$; as explained at the beginning of
Section~\ref{sectGaldefdegp}, we can assume that $\eta$ is a 1-unit.

\pars
Take $a\in L\setminus K$. The proof of the following fact can be found in 
\cite[Section 6.3]{KKS}.
\begin{lemma}                             \label{trace}
Take a separable field extension $K(a)|K$ of degree $n$ and let $f(X)\in K[X]$ be the 
minimal polynomial of $a$ over $K$. Then
\begin{equation}
\tr_{K(a)|K} \left(\frac{a^m}{f'(a)}\right)\>=\> \left\{
\begin{array}{ll}
0 & \mbox{if } 1\leq m\leq n-2\\
1 & \mbox{if } m=n-1\>.
\end{array}
\right.
\end{equation}
\end{lemma} \qed

\pars
For arbitrary $b,c\in K$, we note:
\begin{equation}                                \label{da-c}
b(a-c)^{p-1}\,\in\cM_{K(a)} \>\Longleftrightarrow\> vb\,>\,-(p-1)v(a-c)\>.
\end{equation}

\pars
First we consider Artin-Schreier extensions. By
Lemma~\ref{trace} and equation~(\ref{vf'AS}),
\begin{equation}                     \label{trtheta-c}
\tr_{K(\vartheta)|K}\left((\vartheta-c)^i\right)\>=\> \left\{
\begin{array}{ll}
0 & \mbox{if } 1\leq i\leq p-2\\
-1 & \mbox{if } i=p-1\>
\end{array}
\right.
\end{equation}
for arbitrary $c\in K$ since $\vartheta-c$ is also an Artin-Schreier generator. In
particular,
\[
\tr_{K(\vartheta)|K}  \left(b(\vartheta-c)^{p-1}\right) \>=\> -b
\]
for all $b\in K$. By (\ref{da-c}) it follows that
\begin{eqnarray}
\quad\tr_{K(\vartheta)|K} \left(\cM_{K(\vartheta)}\right) & \supseteq &
\{b\in K \mid vb>-(p-1)v(\vartheta-c)\mbox{ for some } c\in K\}\label{trMe}\\
&=& \{b\in K \mid vb\geq -(p-1)v(\vartheta-c)\mbox{ for some } c\in K\}\nonumber\\
&=& \{b\in K \mid vb\in -(p-1)v(\vartheta-K)\}\nonumber\\
&=& \{b\in K \mid vb\in (p-1)\Sigma_\cE\}\>, \nonumber
\end{eqnarray}
where the first equality follows from the fact that $-(p-1)v(\vartheta-c)$ has no 
smallest element, and the last equality follows from equation (\ref{SigmaAS}) of
Theorem~\ref{dist_galois_p}.

\parm
Now we consider Kummer extensions. Since $(\eta^i)^p\in K$, we have that
\[
\tr_{K(\eta)|K}(\eta^i)\>=\> 0
\]
for $1\leq i\leq p-1$. For $c\in K$ and $0\leq j\leq p-1$, we compute:
\[
(\eta-c)^j \>=\> \sum_{i=1}^j \binom{j}{i} \eta^i (-c)^{j-i}\>+\> (-c)^j\>.
\]
Thus for every $b\in K$,
\begin{equation}                      \label{treta-c}
\tr_{K(\eta)|K} (b(\eta-c)^j) \>=\> pb(-c)^j\>.
\end{equation}
By (\ref{da-c}), $b(\eta-c)^{p-1}\in\cM_{K(\eta)}$ holds if and only if 
$vb>-(p-1)v(\eta-c)$; 
the latter remains true if we make $v(\eta-c)$ even larger. Since $\eta$ is a $1$-unit, 
there is $c\in K$ such that $v(\eta-c)>0$, which implies that $vc=0$. Hence we may 
choose $c\in K$ with $vb>-(p-1)v(\eta-c)$ \ {\it and} \ $\>vc=0$. Applying
(\ref{treta-c}) with $j=p-1$, we find that $\tr_{K(\eta)|K} (b(-c)^{-(p-1)}
(\eta-c)^{p-1})=pb$. We obtain:
\begin{eqnarray}                                     \label{trMm}
\qquad\tr_{K(\eta)|K}  \left(\cM_{K(\eta)}\right) &\supseteq&
\{pb\mid b\in K,\, vb>-(p-1)v(\eta-c)\mbox{ for some } c\in K\}\\
&=& \{b\in K\mid vb\geq vp-(p-1)v(\eta-c)\mbox{ for some } c\in K\}\nonumber\\
&=& \left\{b\in K \mid vb\in (p-1)\left(\frac{1}{p-1}vp-v(\eta-K)\right)
\right\} \nonumber\\
&=& \{b\in K \mid vb\in (p-1)\Sigma_\cE\}\>, \nonumber
\end{eqnarray}
where the first equality follows from the fact that $-(p-1)v(\vartheta-c)$ has no 
smallest element, and the last equality follows from equation (\ref{SigmaKum}) of
Theorem~\ref{dist_galois_p}.

\parm
In order to prove the opposite inclusions in (\ref{trMe}) and (\ref{trMm}), we have 
to find out enough information about the elements $g(a)
%=a_{p-1}a^{p-1}+a_{p-2}a^{p-2}+\ldots+a_0
\in K(a)$ that lie in
$\cM_{K(a)}$. Using the Taylor expansion, we write
\[
g(a)\>=\>\sum_{i=0}^{p-1} \partial_i g(c)(a-c)^i\>.
\]
By Lemma~\ref{Kap} there is $c\in K$ such that among the values $v\partial_i g(c)
(a-c)^i$, $0\leq i\leq p-1$, there is precisely one of minimal value, and the same 
holds for all $c'\in K$ with $v(a-c')\geq v(a-c)$. In particular, we may assume that
$v(a-c)>va$. For all such $c$, we have:
\[
vg(a)\>=\>\min_{0\leq i\leq p-1} v\partial_i g(c)(a-c)^i\>.
\]
Hence for $g(a)$ to lie in $\cM_{K(a)}$ it is necessary that $v\partial_i g(c)(a-c)^i>0$, 
or equivalently,
\begin{equation}                      \label{vgic}
v\partial_i g(c) \> >\> -iv(a-c)
\end{equation}
for $0\leq i\leq p-1$ and $c\in K$ with $v(a-c)>va$.

\parm
In the Artin-Schreier extension case, for $g(\vartheta)\in\cM_{K(\vartheta)}$ and 
$c\in K$ with $v(a-c)>va$, using (\ref{trtheta-c}) we find:
\[
\tr_{K(\vartheta)|K} (g(\vartheta))\>=\> \sum_{i=0}^{p-1} \tr_{K(\vartheta)|K}
(\partial_i g(c)(\vartheta-c)^i) \>=\> -\partial_{p-1}g(c)\>.
\]
Since $\partial_{p-1}g(c)>-(p-1)v(\vartheta-c)$ by (\ref{vgic}), this proves the 
desired equality in (\ref{trMe}).

\parm
In the Kummer extension case, for $g(\eta)\in\cM_{K(\eta)}$ and $c\in K$ with $v(a-c)>va$, 
using (\ref{treta-c}) we find:
\[
\tr_{K(\eta)|K} (g(\eta))\>=\> \sum_{j=0}^{p-1} \tr_{K(\eta)|K}(\partial_j 
g(c)(\eta-c)^j) \>=\> p\sum_{j=0}^{p-1} \partial_j g(c)(-c)^j\>.
\]
As we assume that $v(\eta-c)>0$, we have that $vc=0$ and 
\[
-iv(\eta-c)\geq -(p-1)v(\eta-c)\quad\mbox{ for }\> 0\leq i\leq p-1\>.
\]
Hence by (\ref{vgic}), $v \sum_{i=0}^{p-1} \partial_i g(c)(-c)^i\geq -(p-1)v(\eta-c)$.
This proves the desired equality in (\ref{trMm}). 

\mn
\begin{remark}                              \label{rem}
In the case of a Kummer extension, it follows from (\ref{trMm}) that $p\cO_K\subseteq
\tr_{K(\eta)|K} \left(\cM_{K(\eta)}\right)$ since $vp=vp-(p-1)v\eta$ as $\eta$ is a 
unit.
\end{remark}

\mn
%
%--------------------------------------------------------------------------
%
\subsection{Proof of Theorem~\ref{defetrace}}                    \label{secttracepf}
\mbox{ }\sn
We have just proved the second equality of (\ref{defetraceeq}) in the previous section; 
the third equality follows from part 1) of Lemma~\ref{lemI^n=I}. The first equality is 
seen as follows. Take any $a\in \cO_L\,$. As the defect extension $(L|K,v)$ is immediate,
there is some $b\in\cO_K$ such that $a-b\in\cM_L$. Then $\tr_{L|K}(a)=\tr_{L|K}(b)+
\tr_{L|K}(a-b)$ with $\tr_{L|K}(a-b)\in\tr_{L|K}(\cM_L)$. If $\chara L=p$, then 
$\tr_{L|K}(b)=0$. If $\chara L=0$, then $\tr_{L|K}(b)=pb\in p\cO_K$, which is contained 
in $\tr_{L|K}(\cM_L)$ as stated in Remark~\ref{rem}.

\pars
Assume that the extension $\cE$ has independent defect. Then by part 1) of
Proposition~\ref{ind=I^p=I}, $I_\cE^p=I_\cE$ as well as $(I_\cE\cap K)^p=
I_\cE\cap K$. This implies that $I_\cE^{p-1}=I_\cE$ and $(I_\cE\cap K)^{p-1}=
I_\cE\cap K$. On the other hand, we know from the equivalence of statements a) and b) 
in Theorem~\ref{eqindep} that $I_\cE=\cM_{v_H}\,$, where $H$ is a strongly convex 
subgroup of $vL$, and that $H$ is equal to the convex subgroup $H_\cE$ appearing in the
definition of independent defect. Now the second equation of (\ref{trMM}) follows from
(\ref{defetraceeq}) as $(I_\cE\cap K)^{p-1}=I_\cE\cap K=\cM_{v_H}\cap K$. The third 
equation of (\ref{trMM}) holds since $\cM_{v_H}=(a\in L \mid va>H)$. 

Now assume that (\ref{trMM}) holds for some strongly convex subgroup $H$ of $vL=vK$. 
Set $\Sigma:=vL\setminus H=vK\setminus H$. From (\ref{defetraceeq}) we now obtain:
\[
(b\in K \mid vb\in (p-1)\Sigma_\cE)\>=\>(b\in K\mid vb\in \Sigma)\>.
\]
On the one hand, we have
\[
(b\in K \mid vb\in (p-1)\Sigma_\cE)\>=\>(b\in K \mid vb\in ((p-1)\Sigma_\cE)\fs)\>,
\]
and on the other, 
\[
(b\in K\mid vb\in \Sigma) \>=\> (b\in K\mid vb\in ((p-1)\Sigma)\fs)
\]
by Lemma~\ref{SD}. This implies that $((p-1)\Sigma_\cE)\fs=((p-1)\Sigma)\fs$.
Suppose that $\Sigma_\cE\ne\Sigma$. If there is $\alpha\in\Sigma_\cE\setminus\Sigma$,
then $\alpha<\Sigma$, whence $(p-1)\alpha<(p-1)\Sigma$ and therefore, 
$((p-1)\Sigma_\cE)\fs\supsetneq ((p-1)\Sigma)\fs$, contradiction. Symmetrically, we 
obtain a contradiction if there is $\alpha\in\Sigma\setminus\Sigma_\cE\,$. 
This shows that $\Sigma=\Sigma_\cE\,$, and we thus obtain that $\cE$ has independent 
defect and that $H=H_\cE\,$.

\pars
The last statement of the theorem holds since $vp\notin H_\cE=H$ by part 2) of
Proposition~\ref{ind=I^p=I}. If $(vK)_{vp}$ is archimedean, this forces $H=\{0\}$.
\qed

\mn
%
%
%--------------------------------------------------------
%
\subsection{Differents}          \label{sectDd}
\mbox{ }\sn
In this section, we compute the differents $\cD(\cO_L|\cO_K)$ in the case where
$vL=vK$. We will employ the following auxiliary results.
\begin{lemma}                              \label{O:I}
Take any $\cO_L$-fractional ideal $I$. 
\sn
1) If $vI$ has an infimum $\alpha$ in $vL$, then $\cO_L :_L I = (b\in L\mid vb\geq
-\alpha)$. If $vI$ has no infimum in $vL$, then $\cO_L :_L I = (b\in L\mid vb>-vI)$, 
which is not a principal ideal, and $v(\cO_L :_L I)$ has no infimum. 
\sn
2) If $vI$ has an infimum $\alpha$ in $vL$, then $\cO_L :_L (\cO_L :_L I) = (b\in L\mid
vb\geq\alpha)$. If $vI$ has no infimum in $vL$, then $\cO_L :_L (\cO_L :_L I) = I$. 
\end{lemma}
\begin{proof}
1): We compute:
\begin{eqnarray*}
\cO_L :_L I&=& (b\in L\mid bI\subseteq\cO_L)\>=\> (b\in L\mid vb+\beta\geq 0 \mbox{ for 
all } \beta\in vI)\\
&=& (b\in L\mid vb\geq -vI)\>.
\end{eqnarray*}
Since $\alpha$ is an infimum of $vI$ if and only if $-\alpha$ is a supremum of $-vI$, 
this yields our assertions.
\sn
2): This follows by applying part 1) twice.
\end{proof}

Denote by $\cT(\cO_L|\cO_K)$ the $\cO_L$-ideal generated 
by $\tr_{L|K}(\cO_L)$. We use the previous lemma to show:
\begin{theorem}                            \label{prop}
Assume that $(L|K,v)$ is a finite unibranched extension with $vL=vK$. Then we have:
\sn
1) $\cC(\cO_L|\cO_K)=\cO_L:_L \cT(\cO_L|\cO_K)$ and 
$\cD(\cO_L|\cO_K)=\cO_L:_L (\cO_L:_L \cT(\cO_L|\cO_K))$.
\sn
2) If $v\tr_{L|K}(\cO_L)$ has an infimum $\alpha$ in $vK$, then $\cD(\cO_L|\cO_K) = 
(b\in L\mid vb\geq\alpha)$ and if $\cD(\cO_L|\cO_K)\ne \cT(\cO_L|\cO_K)$, then
\begin{equation}                               \label{MDsubT}
\cM_L\, \cD(\cO_L|\cO_K)\>=\> \cT(\cO_L|\cO_K)\>.
\end{equation} 
If $v\tr_{L|K}(\cO_L)$ has no infimum 
in $vK$, then $\cD(\cO_L|\cO_K) = \cT(\cO_L|\cO_K)$, which is not principal.
\end{theorem}
\begin{proof} 
1): Take any $z\in L$. Since $vL=vK$, there is $z_1\in K$ such that $vz_1=vz$, so
that $z=z_1 z_2$ with $z_2$ a unit in $\cO_L\,$. Then 
\[
\tr_{L|K}(z\cO_L)\>=\>z_1\tr_{L|K}(z_2\cO_L)\>=\>z_1\tr_{L|K}(\cO_L) \>.
\]
Observe that $z_1 \tr_{L|K}(\cO_L)\subseteq \cO_K$ if and only if $z\cT(\cO_L|\cO_K)
\subseteq \cO_L\,$. This proves the assertions of part 1).
\sn
2): A part of the assertions follow from part 1) together with part 2) of 
Lemma~\ref{O:I}; it only remains to prove (\ref{MDsubT}) in case $\cD(\cO_L|\cO_K)$
properly contains $\cT(\cO_L|\cO_K)$. In this case, $\cD(\cO_L|\cO_K)=
(b\in L\mid vb\geq \alpha)$ and $\cT(\cO_L|\cO_K)=(b\in L\mid vb> \alpha)=v\cM_L+
(b\in L\mid vb\geq \alpha)$, which proves (\ref{MDsubT}).
\end{proof}

We will also need the following simple observation.
\begin{lemma}                      \label{(IcapK)}
Take an extension $(L|K,v)$ of valued fields with $vL=vK$, an $\cO_L$-ideal $I$, and
$n\in\N_{>0}\,$. Then the $\cO_L$-ideal $J$ generated by $(I\cap K)^n$ equals $I^n$.
\end{lemma}
\begin{proof}
The inclusion $J\subseteq I^n$ is obvious. Now take any $b\in I$. Since $vL=vK$, there is
$c\in K$ such that $vc=vb$, which implies that $c\in I\cap K$. Now $b^n=(b/c)^n c^n$ with
$(b/c)^n\in\cO_L$ and $c^n\in I\cap K$. Therefore, $b^n\in J$, which proves the reverse
inclusion.
\end{proof}

\sn
Proof of Theorem~\ref{Dd}: 
\sn
1): Lemmas~\ref{vderimmAS} and~\ref{vderimmK} together with Proposition~\ref{VdO} show 
that i) and ii) are equal. Proposition~\ref{VdO} shows that ii) and iii) are equal. 
From Theorem~\ref{defetrace} we know that $\tr_{L|K}\left(\cO_L\right) = (I_\cE\cap K)
^{p-1}$, whence $\cT(\cO_L|\cO_K)=I_\cE^{p-1}$ by Lemma~\ref{(IcapK)}. This shows that 
i) and iv) are equal.
\sn
2): If $v\cT(\cO_L|\cO_K)=vI_\cE^{p-1}$ has no infimum in $vK=vL$,
then from Theorem~\ref{prop} we obtain that $\cD(\cO_L|\cO_K) = \cT(\cO_L|\cO_K)
=I_\cE^{p-1}$. As the extension $(L|K,v)$ is immediate, $vI_\cE^{p-1}$ has no smallest 
element. Hence if $vI_\cE^{p-1}$ has an infimum $\alpha$, it does not contain it, and 
again by Theorem~\ref{prop}, $\cD(\cO_L|\cO_K) =(b\in L\mid vb\geq\alpha)\supsetneq
I_\cE^{p-1}$ and $\cM_L\, \cD(\cO_L|\cO_K)= \cT(\cO_L|\cO_K)=I_\cE^{p-1}$. Note that if
$\alpha=va$ for $a\in L$, then $(b\in L\mid vb\geq\alpha)=a\cO_L\,$.
\sn
3): Assume that $(K,v)$ has rank $1$. Then the only proper convex subgroup of $vL$ is 
$\{0\}$. Thus Proposition~\ref{propann} together with Theorems~\ref{OASdef} 
and~\ref{OKumdef} shows that $\ann\Omega_{\cO_L|\cO_K}=I_\cE^{p-1}$ if
$vI_\cE^{p-1}$ has no infimum in $vL$, and $\ann\Omega_{\cO_L|\cO_K}=(b\in L\mid
vb\geq\alpha)$ if $vI_\cE^{p-1}$ has infimum $\alpha$. From what we have already 
shown, it follows that $\cD(\cO_L|\cO_K) =\ann\Omega_{\cO_L|\cO_K}\,$.
\sn
4): Assume that the extension $\cE$ has independent defect; then $\Omega_{\cO_L|\cO_K}=0$ 
by Theorem~\ref{eqindep} and therefore, $\ann\Omega_{\cO_L|\cO_K}=\cO_L\,$. If $H_\cE=
\{0\}$, then $\tr_{L|K}\left(\cO_L\right) = \tr_{L|K} \left(\cM_L\right) = \cM_K$ by
Theorem~\ref{defetrace}.
%, hence $I_\cE^{p-1}=\cT(\cO_L|\cO_K)=\cM_L\,$. 
Since $H_\cE=\{0\}$ is a strongly convex subgroup of $vL=vK$, we know that $vK$ has no 
minimal positive element. So $v\cM_K$ has infimum $0$ and it follows from part 2) of
Theorem~\ref{prop} that $\cD(\cO_L|\cO_K)=\cO_L\,$. 
%Again, Proposition~\ref{ann} shows that $\ann\Omega_{\cO_L|\cO_K}=\cO_L\,$.

Now assume that $H_\cE\ne\{0\}$. Since 
%$vL=vK$, we also have $v_{H_\cE}L=v_{H_\cE}K$, and since 
$\cM_{v_{H_\cE}}\cap K=\cM_{v_{H_\cE}|_K}\,$, Lemma~\ref{(IcapK)} shows that the 
$\cO_L$-ideal generated by $\cM_{v_{H_\cE}}\cap K$ is $\cM_{v_{H_\cE}}\,$.
By Theorem~\ref{defetrace}, $\tr_{L|K}\left(\cO_L\right) = \tr_{L|K} \left(\cM_L\right)
=\cM_{v_{H_\cE}}\cap K$, so $\cT(\cO_L|\cO_K)=\cM_{v_{H_\cE}}$. As
$v(\cM_{v_{H_\cE}}\cap K)=v\cM_{v_{H_\cE}}$ has no infimum in $vK$, part 2) of
Theorem~\ref{prop} shows that  
$\cD(\cO_L|\cO_K)=\cT(\cO_L|\cO_K)=\cM_{v_{H_\cE}}\subsetneq\cO_L=
\ann\Omega_{\cO_L|\cO_K}$ in this case.

Finally, assume that $\cD(\cO_L|\cO_K)$ is equal to $\cO_L$ or to $\cM_{v_H}$ for a 
strongly convex nontrivial subgroup $H$ of $vL$. If the former holds, then 
$\cD(\cO_L|\cO_K)$ is not equal to $I_\cE^{p-1}$ since $vI_\cE=\Sigma_\cE$ and hence also 
$I_\cE^{p-1}$ have no smallest element. Hence from the proof of part 2) 
it follows that $0$ is the infimum of $vI_\cE^{p-1}$. This implies that $I_\cE=\cM_L
=\cM_{v_H}$ where $H=\{0\}$. Since $v\cM_L=vI_\cE$ has no smallest element, it follows 
that $\{0\}$ is a strongly convex subgroup of $vL$.

If the latter holds, then $\cD(\cO_L|\cO_K)$ is not principal and therefore by part 2), 
$\cD(\cO_L|\cO_K)$ can only be equal to $I_\cE^{p-1}$. This implies that
$I_\cE=\cM_{v_H}$ for some strongly convex subgroup of $vL$. In both cases, it follows 
from the equivalence a)$\Leftrightarrow$b) of Theorem~\ref{eqindep} that $\cE$ has
independent defect.
\qed

\bn
%
%--------------------------------------------------------
%
\section{Proof of Theorems~\ref{eqindep2AS} and~\ref{eqindep2K}}  \label{sectpfthm1.4}
Take a defect extension $\cE=(L|K,v)$ of degree $p$. Then by equation (\ref{feuniq}),
the extension is unibranched and immediate, so in particular, $vL=vK$. 

We first prove Theorem~\ref{eqindep2AS}, so we assume that $\cE$ is an 
Artin-Schreier extension generated by $\vartheta\in L$ with $\vartheta^p-\vartheta=
a\in K$. By Corollary~\ref{SigmaE-}, $v(\vartheta-c)<0$ for all $c\in K$.
It follows that $v(\vartheta^p-c^p)=v(\vartheta-c)^p=pv(\vartheta-c)<v(\vartheta-c)$ 
and thus,
\begin{equation}                           \label{v(a-c^p+c)}
v(a-\wp(c))\>=\>v(\vartheta^p-\vartheta-c^p+c)\>=\>\min\{v(\vartheta^p-c^p),
v(\vartheta-c)\}\>=\>pv(\vartheta-c)\>.
\end{equation}
This proves equation~(\ref{1.4eq1}).

\pars
The equivalence a)$\Leftrightarrow$b) follows from the definition of independent defect
together with equation~(\ref{SigmaAS}) of Theorem~\ref{dist_galois_p} which says that 
$\Sigma_\cE=-v(\vartheta-K)$. The equivalences b)$\Leftrightarrow$c) and 
e)$\Leftrightarrow$f) follow from equation~(\ref{1.4eq1}). 
The equivalence c)$\Leftrightarrow$d) is trivial.

\pars
Now assume that $vK$ is $p$-divisible. Then the equivalence a)$\Leftrightarrow$e) 
follows from part 1) of Proposition~\ref{ind=I^p=I} together with the equation 
$\Sigma_\cE=-v(\vartheta-K)$.

\pars
Finally, assume that $(K,v)$ has rank $1$. Then the only proper convex subgroup of $vK$ 
is $H=\{0\}$, so that condition d) is equivalent to condition g).

\parb
Now we prove Theorem~\ref{eqindep2K}, so we assume that $\cE$ is a Kummer
extension. As explained at the beginning of Section~\ref{sectGaldefdegp}, it is 
generated by a 1-unit $\eta\in L$ with $\eta^p=a\in K$. By 
Corollary~\ref{SigmaE-}, $v(\eta-c)<\frac{1}{p-1}vp$ for all $c\in K$.

\pars
The following is part of Lemma~2.18 of \cite{KuRz}:
\begin{lemma}                                               \label{eta-c}
Take $\eta\in\tilde{K}$ such that $\eta^p\in K$ and $v\eta=0$. Then for $c\in K$ 
such that $v(\eta-c)>0$, $v(\eta-c)<\frac{1}{p-1}vp$ holds if and only if $v(\eta^p
-c^p)<\frac{p}{p-1}vp$, and if this is the case, then $v(\eta^p-c^p)=pv(\eta-c)$. 
\end{lemma}
In order to prove equation~(\ref{1.4eq2}), we have to discuss the remaining case of
$v(\eta-c)\leq 0$. If $v(\eta-c)=0$, then $vc\geq 0$ and $c$ is not a 1-unit. Hence 
$c^p$ is not a 1-unit while $a=\eta^p$ is, whence $v(a-c^p)=0=pv(\eta-c)$. 
If $v(\eta-c)<0$, then $vc<0$ and $v(\eta-c)=vc$. It follows that $vc^p<0=v\eta^p$, so
that $v(\eta^p-c^p)=vc^p=pvc=pv(\eta-c)$. This completes the proof of 
equation~(\ref{1.4eq2}).

\pars
The equivalence a)$\Leftrightarrow$b) follows from the definition of independent defect
together with equation~(\ref{SigmaKum}) of Theorem~\ref{dist_galois_p}, which says that 
$\Sigma_\cE=\frac{1}{p-1}vp-v(\vartheta-K)$. The equivalences b)$\Leftrightarrow$c) and 
e)$\Leftrightarrow$f) follow from equation~(\ref{1.4eq2}) and the fact that $\alpha<H$ 
if and only if $p\alpha<H$ since $H$ is a convex subgroup of $vK$. 
The equivalence c)$\Leftrightarrow$d) is trivial.

\pars
Now assume that $vK$ is $p$-divisible. Then by part 1) of Proposition~\ref{ind=I^p=I}
together with the equation $\Sigma_\cE=\frac{1}{p-1}vp-v(\vartheta-K)$, condition a) is
equivalent to the equation
\[
\frac{p}{p-1}vp-pv(\eta-K) \>=\> \frac{1}{p-1} vp-v(\eta-K)\>.
\]
Multiplying both sides with $-1$ and then adding $\frac{p}{p-1}vp$ to both sides, we 
find that this equation is equivalent to condition e). 

\pars
Finally, if $(K,v)$ has rank $1$, then it follows as in the proof of
Theorem~\ref{eqindep2AS} that condition d) is equivalent to condition g).

%\bn\bn

\end{document}